\documentclass[a4paper]{amsart}
\usepackage{amsmath}
\usepackage{amsthm}
\usepackage{amssymb}
\usepackage{verbatim}
\usepackage[all]{xy}

\newtheorem{Theorem}{Theorem}[section]
\newtheorem{Lemma}[Theorem]{Lemma}
\newtheorem{Proposition}[Theorem]{Proposition}
\newtheorem{Corollary}[Theorem]{Corollary}
\theoremstyle{definition} \newtheorem{Definition}[Theorem]{Definition}
\theoremstyle{remark}\newtheorem{Remark}[Theorem]{Remark}
\theoremstyle{remark}
\newcommand{\A}{\mathop{\mathcal{A}}\nolimits}
\newcommand{\G}{\mathop{\mathcal{G}}\nolimits}
\newcommand{\F}{\mathop{\mathcal{F}}\nolimits}
\newcommand{\D}{\mathop{\mathcal{D}}\nolimits}
\newcommand{\N}{\mathop{\mathcal{N}}\nolimits}
\newcommand{\T}{\mathop{\mathcal{T}}\nolimits}
\newcommand{\E}{\mathop{\mathcal{E}}\nolimits}
\newcommand{\Ol}{\mathop{\mathcal{O}}\nolimits}

\newcommand{\V}{\mathop{\mathcal{V}}\nolimits}
\newcommand{\pr}{\mathop{\textrm{pr}}\nolimits}
\newcommand{\Hom}{\mathop{\textrm{Hom}}\nolimits}
\newcommand{\Sing}{\mathop{\textrm{Sing}}\nolimits}
\newcommand{\Res}{\mathop{\textrm{Res}}\nolimits}

\newcommand{\id}{\mathop{\textrm{id}}\nolimits}
\newcommand{\codim}{\mathop{\textrm{Codim}}\nolimits}

\begin{document}
\title{Partial holomorphic connections and extension of foliations}
\author{Isaia Nisoli}
\address{Dipartimento di Matematica, Largo B. Pontecorvo 5, I-56127 Pisa \\ Office +39 050 221 3237
Fax +39 050 221 3224}
\email{nisoli@mail.dm.unipi.it}

\begin{abstract}
This paper stresses the strong link between the existence of partial holomorphic
connections on the normal bundle of a foliation seen as a quotient of the ambient tangent bundle and
the extendability
of a
foliation to an infinitesimal neighborhood of
a submanifold. We find the obstructions to extendability and thanks to the theory developed we
obtain some new
Khanedani-Lehmann-Suwa type index theorems.
\end{abstract}

\maketitle

\section*{Introduction}
Localization of characteristic classes is an important tool in differential geometry, topology and dynamics in
particular for complex dynamical systems  \cite{CamSadAnn}.
In this context many different indexes have been developed during the years: among them the Baum-Bott and the
Camacho-Sad indexes. A global framework for this theory has been provided by Suwa and Lehmann \cite{Suwa}: the
fundamental principle is that the
existence of a flat partial holomorphic connection (called a holomorphic action in \cite{Suwa}) implies the
vanishing of the Chern classes associated to some vector bundles.
Suppose we are working on a compact manifold $M$ and we have a partial holomorphic connection outside an analytic
subset $\Sigma$ of $M$. We can localize these Chern classes to $\Sigma$ and, using Poincar\'{e} and Alexander
duality, define the residue of the characteristic class at $\Sigma$ (a short account of this process is
given in section \ref{SectionVanishing}).

Now, at least two different research directions arise: to adapt such a theory to singular manifolds and submanifolds
\cite{LeSu1}, \cite{LeSu2}, or to try to develop new vanishing theorems.
This paper falls into the second group. As we said such vanishings theorems arise when we have the existence of partial
holomorphic connections; this is the case when we have a holomorphic foliation which leaves a submanifold $S$ invariant.
This gives rise to index theorems for $N_{\F|_S}$, the normal bundle of the foliation seen as a
quotient of the tangent bundle of the submanifold (Baum-Bott index), $N_S$, the normal bundle to the submanifold
(Camacho-Sad index) and $N_{\F}|_S$ the normal bundle of the foliation seen as a quotient of the ambient tangent bundle
restricted to $S$ (Kahnedani-Lehmann-Suwa or variation index \cite{KaSu},\cite{LeSu1}). 
The fundamental reference on all
these topics is \cite{Suwa}.

The same tecniques allow to prove other index theorems, even if the holomorphic foliation is
transverse to the submanifold, as the index theorem for the bundle $\Hom(\F,N_S)$, which gives rise
to the tangential index \cite{Honda}. 

In the last years, a new theory was developed also for endomorphisms of a complex manifold leaving a submanifold
pointwise invariant \cite{ABTHolMap} and the case of foliation transverse to a submanifold in the Camacho-Sad and
Baum-Bott case \cite{ABTHolMapFol},\cite{Cam},\cite{CaLe},\cite{CaMoSa}.
The key to the existence of partial holomorphic connections is the vanishing of the Atiyah class, a cohomological
obstruction to the splitting of a short exact sequence of sheaves of $\Ol_S$-modules \cite{Atiyah}.
In the paper \cite{ABTHolMapFol} the Atiyah sheaf for the normal bundle of a submanifold was described in a more
concrete way, giving new insights to the problem.
Further developments as \cite{ABTEmbeddings} showed the strong connection between the existence of partial holomorphic
 connections for $N_S$ and the ``regularity'' of the embedding of a subvariety.

In section \ref{SecAtiyah} of this paper we find a more concrete realization of the Atiyah sheaf for the normal bundle
of a foliation seen as a quotient of the ambient tangent bundle and study some sufficient conditions for the existence
of a
more general variation action.
First of all in section \ref{Sec:Frobenius} we define what a foliation of the $k$-th infinitesimal neighborhood of a
submanifold is and prove some Frobenius type theorems for such foliations, which give us the possibility of choosing
atlases with some particular
structure; in these special atlases, it is clear that the existence of a foliation of the first
infinitesimal neighborhood is the key to
the existence of partial holomorphic connections on the normal bundle of a foliation seen as a quotient of the ambient
tangent bundle.
Therefore, to generalize the variation index we have to find foliations of the first infinitesimal neighborhood: with
this aim we study the problem of how to ``project'' a transversal foliation to a tangential one,
using first order
splitting (section \ref{Sec2Split}) and how to extend a foliation of a submanifold $S$ to an infinitesimal
neighborhood (section \ref{SecExtension}).
Moreover, thanks to the new realization of the Atiyah sheaf, we develop in section \ref{SectionNonInvolutive} a result
about non-involutive subsheaves of $\T_S$ which extend to the first infinitesimal neighborhood. This gives us
informations about vanishing of the characteristic classes of the involutive closure of their restriction to $S$ (the
smallest involutive subsheaf of $\T_S$ containing it) and some more results regarding the extension problem.
Thanks to the machinery developed we can then prove some new index theorems, generalizing the Khanedani-Lehmann-Suwa
action and compute their indexes is some simple cases.

\noindent \textbf{Acknowledgements. }The article is part of the author's Ph. D. thesis work and he
would like to thank Professor Marco Abate, his advisor, 
for his thoughtful guide and many useful advices and hints, Professor Tatsuo Suwa for many precious conversations, his
patience and his wisdom, and Professor Filippo Bracci for an important suggestion.
The author would like to thank Prof. Kioji Saito and the Institute for the Physics and Mathematics of the Universe,
Kashiwa, Japan and the
International Centre for Theoretical Physics, Trieste, Italy for the warm hospitality and wonderful research enviroments
offered to him.
The author would like also to thank GNSAGA, for the help in funding his mission to Japan.

\section*{Notation and conventions}
In this paper we are going to use the Einstein summation convention. To ease the understanding of the computations the
indexes are going to have a fixed range.
In this paper, $M$ is a $n$-dimensional complex manifold, $S$ a complex subvariety of codimension $m$ and $\F$ a
dimension $l$ holomorphic foliation of either $M$ or $S$, with $l\leq n-m$. Then the indexes are going to have the
following range:
\begin{itemize}
\item $h,k$ will range in $1,\ldots, n$, these are the indexes relative to the coordinate system of $M$;
\item $p,q$ will range in $m+1,\ldots,n$, in an atlas adapted to $S$ (see definition \ref{Def:AdaptedS}); these are the
indexes relative to the coordinates along $S$;
\item $r,s$ will range in $1,\ldots,m$, in an atlas adapted to $S$; these are the indexes relative to the
coordinates normal to $S$;
\item $i,j$ will range in $m+1,\ldots,m+l$, in an atlas adapted to $\F$ (see definition \ref{Def:AdaptedFS}); these are
the indexes relative to the coordinates along $\F$;
\item $u,v$ will range in $1,\ldots,m,m+l+1,\ldots,n$, in an atlas adapted to $\F$; these are the the indexes relative
to the coordinates normal to $\F$.
\end{itemize}
In case we shall need more indexes of each type, we shall prime $'$ them or put a subscript, e.g. $r_1$.
We shall denote by $\Ol_M$ the structure sheaf of holomorphic functions on $M$, by $\mathcal{I}_S$ the ideal sheaf of a
subvariety $S$ and by $\mathcal{I}_S^k$ its $k$-th power as an ideal.
If $f$ is an element of $\Ol_M$ we will denote by $[f]_{k+1}$ its image in $\Ol_{S(k)}:=\Ol_M/\mathcal{I}_S^{k+1}$.
Moreover we denote by $\T_M$ and $\T_S$ the tangent sheaves to $M$ and $S$ respectively, where defined.
The following is a definition we will use through the whole paper.
\begin{Definition}\label{Def:AdaptedS}
Let $\mathcal{U}$ be an atlas for $M$. We say that $\mathcal{U}$ is \textbf{adapted to $S$} if on each coordinate
neighborhood
$(U_{\alpha},z^1_{\alpha},\ldots,z^n_{\alpha})$ such that $U\cap S$ is not empty, we have that $S\cap
U_{\alpha}=\{z^1_{\alpha}=\ldots=z^m_{\alpha}=0\}$, where $m$ is the codimension of $S$.
\end{Definition}

\section{Foliations of $k$-th infinitesimal neighborhoods}\label{Sec:Frobenius}
In this section we shall define and develop a theory for foliations of $k$-th infinitesimal neighborhoods. We are going
to use the notion of logarithmic vectors field, introduced in \cite{Saito}. The sheaf of these vector
fields can behave badly if the subvariety $S$ they leave invariant is non regular. In the rest of the section $S$ is
assumed to be a submanifold.

\begin{Definition}
The \textbf{$k$-th infinitesimal neighborhood} of a complex submanifold $S$ is the ringed space $(S,\Ol_{S(k)})$, where by $\Ol_{S(k)}$ we denote the quotient sheaf $\Ol_M/\mathcal{I}_S^{k+1}$.
\end{Definition}

\begin{Definition}
A section $v$ of $\T_{M}$ is called \textbf{logarithmic} if $v(\mathcal{\mathcal{I}}_S)\subseteq
\mathcal{\mathcal{I}}_S$.
The sheaf $\T_M(\log S):=\{v\in\T_M\mid v(\mathcal{\mathcal{I}}_S)\subseteq \mathcal{\mathcal{I}}_S\}$ is called the \textbf{sheaf of logarithmic sections} and is a subsheaf of $\T_M$.
The \textbf{tangent sheaf of the $k$-th infinitesimal neighborhood}, denoted by $\T_{S(k)}$, is the image of the sheaf
homomorphism $\T_M(\log S)\otimes_{\Ol_M} \Ol_{S(k)}\to \T_M\otimes_{\Ol_M} \Ol_{S(k)}$ and is a sheaf on $S$.
We will say that a section $v\in \T_{S(k)}$ is \textbf{tangential to the $k$-th infinitesimal neighborhood}.
\end{Definition}
\begin{Remark}\label{StruttLog}
If a point $x$ does not belong to $S$, the stalk $\T_M(\log S)_x$ coincides with $\T_{M,x}$.
Suppose we have an atlas adapted to $S$, if $x\in S$ the stalk $\T_M(\log S)_x$ is generated by
\[ z^r\frac{\partial}{\partial z^s},\frac{\partial}{\partial z^p}.\]
Then a section $v$ of $\T_{S(k)}$ is written locally as:
\[v=[a^r]_{k+1}\frac{\partial}{\partial z^r}+[a^p]_{k+1}\frac{\partial}{\partial z^p},\]
where the $a^r$ belong to $\mathcal{I}_S$.
\end{Remark}
\begin{Remark}
In the following, given a section $v$ of $\T_{S(k)}$ and an open set $U_{\alpha}$ of $M$ intersecting $S$, we shall
denote by $\tilde{v}_{\alpha}$ a local extension of $v$ to $U_{\alpha}$ as a section of $\T_M(U_{\alpha})$; given an
atlas adapted to $S$ it is possible to build such an extension on each coordinate chart. If the open set is clear from
the discussion we shall denote the extension simply by $\tilde{v}$; please note that such an extension is not only a
section of $\T_M(U_{\alpha})$ but also a section of $\T_M(\log S)(U_{\alpha})$. Taken an extension $\tilde{v}$, denoted
by $[1]_{k+1}$ the class of $1$ in $\Ol_{S(k)}(U_{\alpha})$ we shall denote its restriction to the $k$-th infinitesimal
neighborhood by:
\[\tilde{v}\otimes[1]_{k+1}.\] 
We will prove in Lemma \ref{Lemma:Logaritmici} that this notation is consistent with the fact that the sections of
$\T_{S(k)}$ act as derivation of $\Ol_{S(k)}$.
Moreover given two open sets $U_{\alpha}$ and $U_{\beta}$ such that $U_{\alpha}\cap U_{\beta}\cap S\neq\emptyset$ and taken two extension $\tilde{v}_{\alpha}$ and $\tilde{v}_{\beta}$ respectively it follows from the definition that on $U_{\alpha}\cap U_{\beta}$ we have the following equivalence:
\begin{equation}\label{equivideale}
v=\tilde{v}_{\alpha}\otimes[1]_{k+1}=\tilde{v}_{\beta}\otimes[1]_{k+1}.
\end{equation}
\end{Remark}
\begin{Lemma}\label{Lemma:Logaritmici}
The sections of $\T_{S(k)}$ act as derivations of $\Ol_{S(k)}$. Furthermore, given two sections $v,w$ of $\T_{S(k)}$, their bracket, defined on each coordinate patch $U_{\alpha}$ such that $U_{\alpha}\cap S\neq\emptyset$  as
\[[v,w]:=[\tilde{v}_{\alpha},\tilde{w}_{\alpha}]\otimes[1]_{k+1},\]
where the bracket on the right side is the usual bracket on $\T_M$, is a well defined section of $\T_{S(k)}$.
\end{Lemma}
\begin{proof}
Let $v$ be a section of $\T_{S(k)}$ and $f$ a section of $\Ol_{S(k)}$. Let $U_{\alpha}$ and $U_{\beta}$ two coordinate patches of an atlas adapted to $S$ such that $U_{\alpha}\cap U_{\beta}\cap S\neq\emptyset$.
On $U_{\alpha}$ we take representatives $\tilde{f}_1$ and $\tilde{f}_2$ of $f$ and an extension $\tilde{v}_{\alpha}$ of $v$.
We define:
\[v(f):=\tilde{v}_{\alpha}(\tilde{f}_1)\otimes [1]_{k+1}=[\tilde{v}_{\alpha}(\tilde{f}_1)]_{k+1}.\]
We check now that this does not depend on the extension chosen for $f$:
\[\tilde{v}_{\alpha}(\tilde{f}_1)-\tilde{v}_{\alpha}(\tilde{f}_2)=\tilde{v}_{\alpha}(\tilde{f}_1-\tilde{f}_2)=\tilde{v}_{\alpha}(h_{r_1,\ldots,r_{k+1}}z^{r_1}\ldots z^{r_{k+1}}).\]
Since $\tilde{v}_{\alpha}$ is logarithmic, then 
\[\tilde{v}_{\alpha}(h_{r_1,\ldots,r_{k+1}}z^{r_1}\ldots z^{r_{k+1}})\in\mathcal{I}_S^{k+1}\]
and
\[(\tilde{v}_{\alpha}(\tilde{f}_1)-\tilde{v}_{\alpha}(\tilde{f}_2))\otimes[1]_{k+1}=[0]_{k+1}.\]
Now, let $\tilde{v}_{\alpha}$ and $\tilde{v}_{\alpha}'$ be two extensions of $v$. Suppose $w_{1,\alpha},\ldots,w_{n,\alpha}$ are generators for $\T_M(U_{\alpha})$.
By definition:
\[\tilde{v}_{\alpha}-\tilde{v}_{\alpha}'=g^h_{\alpha} w_{h,\alpha},\]
with $g^h_{\alpha}\in \mathcal{I}_S^{k+1}$ for each $h=1,\ldots,n$.
Then:
\[(\tilde{v}_{\alpha}-\tilde{v}_{\alpha}')(\tilde{f}_1)\otimes [1]_{k+1}=g^h_{\alpha}
w_{h,\alpha}(\tilde{f}_1)\otimes
[1]_{k+1}=w_{h,\alpha}(\tilde{f}_1)\otimes [g^h_{\alpha}]_{k+1}=[0]_{k+1}.\]
This implies also that if we take extensions $\tilde{v}_{\alpha}$ and $\tilde{v}_{\beta}$ and representatives $\tilde{f}_{\alpha}$ and $\tilde{f}_{\beta}$ for $f$ on $U_{\alpha}$ and $U_{\beta}$ respectively we have that on $U_{\alpha}\cap U_{\beta}$, the derivation is well defined.

We prove now the bracket is well defined; if $u$ and $v$ are sections of $\T_{S(k)}$ the bracket is:
\[ [u,v]=[\tilde{u},\tilde{v}]\otimes [1]_{k+1}.\]
If $\tilde{u}_1,\tilde{u}_2$ are two extensions of $u$ and $\tilde{v}_1,\tilde{v}_2$ are two extension of $w$
then
\begin{align*}
[\tilde{u}_1,\tilde{v}_1]-[\tilde{u}_2,\tilde{v}_2]&=[\tilde{u}_1,\tilde{v}_1]-[\tilde{u}_1,\tilde{v}_2]+[\tilde{u}_1,
\tilde{v}_2]-[\tilde{u}_2,\tilde{v}_2]\\
&=[\tilde{u}_1,\tilde{v}_1-\tilde{v}_2]+[\tilde{u}_1-\tilde{u}_2,\tilde{v}_2].
\end{align*}
As above, we have that
\[\tilde{u}_1-\tilde{u}_2=g^h_{\alpha}w_{h,\alpha},\quad \tilde{v}_1-\tilde{v}_2=t^h_{\alpha}w_{h,\alpha},\]
with $g^h_{\alpha},t^h_{\alpha}\in\mathcal{I}_S^{k+1}$ for every $h$.
Then:
\begin{align}
[\tilde{u}_1,\tilde{v}_1-\tilde{v}_2]+&[\tilde{u}_1-\tilde{u}_2,\tilde{v}_2]=[\tilde{u}_1,t^h_{\alpha}w_{h,\alpha}]+[
g^h_ { \alpha}w_{h,\alpha},\tilde{v}_2] \nonumber
\\&=\tilde{u}_1(t^h_{\alpha})w_{h,\alpha}+t^h_{\alpha}[\tilde{u}_1,w_{h,\alpha}]-\tilde{v}_2(g^h_{\alpha})w_{h,\alpha}
+g^h_{\alpha}[w_{h,\alpha},\tilde{v}_2].\label{bracbrac}
\end{align}
Since both $\tilde{v}_1$ and $\tilde{u}_2$ are logarithmic, the restriction to the $k$-th infinitesimal neighborhood of
\eqref{bracbrac} is $0$.
\end{proof}
Therefore, the following definition makes sense.
\begin{Definition}\label{Def:FolInf}
A \textbf{regular foliation} of $S(k)$  is a rank $l$ coherent subsheaf $\F$ of $\T_{S(k)}$, such
that:
\begin{itemize}
 \item for every $x\in S$ the stalk $\T_{S(k)}/\F_x$ is $\Ol_{S(k),x}$-free;
 \item for every $x\in S$ we have that $[\F_x,\F_x]\subseteq \F_x$ (where the bracket is the one defined in Lemma
\ref{Lemma:Logaritmici});
 \item the restriction of $\F$ to $S$, denoted by $\F|_S$, is a rank $l$ foliation of $S$.
\end{itemize}
\end{Definition}
\begin{Remark}
Please note that the third condition is a simplifying condition: in the paper \cite{Bracci} a lot of
work is devoted to clarify and explain the concept of extension of a foliation and our definition
is a particular case.
We want to avoid the following situation: let $U$ be an open neighborhood of the origin in
$\mathbb{C}^2$, with
coordinate system $(z^1,z^2)$ and let $S={z^1=0}$. 
We take a subbundle of $\T_{S(1)}$ generated by $[z^1]_2\partial/\partial z^1,\partial/\partial
z^2$. Clearly, it
is involutive with respect to the bracket defined above, but its restriction to $S$ gives rise to a
rank $1$ foliation.
\end{Remark}
The main tool of this section is the Holomorphic Frobenius Theorem, whose proof can be found e.g. in
\cite[pages~38--42]{Suwa}. Lemma \ref{Lemma:Commute} is a tool we use in proving the Frobenius Theorem for foliations of
the $k$-th infinitesimal neighborhood.
We give the proof after a definition.
\begin{Definition}\label{Def:AdaptedFS}
Let $S$ be a codimension $m$ complex submanifold of $M$, a complex $n$-dimensional manifold. Let $\F$ be a rank $l$ regular foliation of $S$. We say that an atlas $\{(U_{\alpha},z^1_{\alpha},\ldots,z^n_{\alpha})\}$ is \textbf{adapted to $S$ and $\F$} if 
\begin{itemize}
 \item $U_{\alpha}\cap S=\{z^1_{\alpha}=\ldots=z^m_{\alpha}=0\},$
 \item $\F|_{U_{\alpha}\cap S}$ is generated by $\partial/\partial z^{m+1}_{\alpha}|_S,\ldots,\partial/\partial z^{m+l}_{\alpha}|_S.$
\end{itemize} 
\end{Definition}
\begin{Remark}
The existence of such an atlas follows from the Holomorphic Frobenius theorem cited above.
\end{Remark}
\begin{Lemma}\label{Lemma:Commute}
Every regular foliation $\F$ of $S(k)$ admits a local frame which can be extended locally by commuting vector fields, i.e.,
for
every point $x\in S$ there exists a neighborhood $U_x$ of $x$ in $M$ and commuting sections 
$\tilde{w}_{m+1},\ldots,\tilde{w}_{m+l}$ of $\T_M$ on $U_x$ such that
$w_{i}:=\tilde{w}_{i}\otimes [1]_{k+1}$ are generators of $\F(U_x\cap S)$. 
\end{Lemma}
\begin{proof}
Let $x$ be a point of $S$, we take a coordinate patch $(U,\phi)$ centered in $x$, adapted to $S$ and $\F|_S$. 
Let $\{v_{i}\}$ be a system of generators of $\F$ in $U\cap S$ and
$\{\tilde{v}_{i}\}$ be vector fields extending them. Call $\mathcal{D}$ the distribution spanned by the
$\tilde{v}_{i}$'s. We complete the frame
$\{\tilde{v}_{i}\}$ to a frame $\{\tilde{v}_{k}\}$ of $\T_M$, taking as $\tilde{v}_t$ the coordinate fields
$\partial/\partial z^t$.
Now, we choose holomorphic functions $f^{k}_{i}$ such that:
\[\tilde{v}_{k}=f_{k}^h\frac{\partial}{\partial z^h}.\]
Please remark that the matrix $A:=(f^{h}_{k})$ is a matrix of holomorphic
functions acting on the right:
\[|\tilde{v}^1,\ldots,\tilde{v}^n|=\bigg|\frac{\partial}{\partial z^1},\ldots,\frac{\partial}{\partial z^n}\bigg|\cdot
A.\]
By hypothesis we know that $A$ is non singular in $x$, so there exists a neighborhood (still denoted by $U$) of $x$
such that this matrix
is invertible
with inverse a matrix of holomorphic functions.
Let $(g^{k}_{h})$ be its inverse matrix. We define $\tilde{w}_{i}=g^{j}_{i}
\tilde{v}_{j}$ and we denote by $w_{i}:=\tilde{w}_{i}\otimes [1]_{k+1}$.
Each one of the $\tilde{w}_{i}$'s belongs to the module generated by
$\tilde{v}_{m+1},\ldots,\tilde{v}_{m+l}$ therefore each $w_i$ belongs to $\T_{S(k)}$. This implies,
thanks to
Lemma \ref{Lemma:Logaritmici} that 
\[
[w_{i},w_{j}]=[\tilde{w}_{i},\tilde{w}_{j}]\otimes[1]_{k+1}=[g^{i'}_i\tilde{v}_{i'},g^{j'}_j\tilde{v}_{j'}]\otimes[1]_{
k+1}\in \F.\]
We claim now that the $\tilde{w}_{j}$ generate $\mathcal{D}$ and therefore, when restricted to $S(k)$ generate $\F$. Let
$\pi$ be
the
projection
$(z^1,\ldots,z^n)\mapsto(z^{m+1},\ldots,z^{m+l})$ and $\Pi=\pi\circ\phi.$ We have:
\[\Pi_*(\tilde{w}_{i})=\Pi_*(\tilde{w}_{i})+g^t_{i}\Pi_*\bigg(\frac{\partial}{
\partial z^t} \bigg)=\Pi_*(g^k_{i}\tilde{v}_{k})=\Pi_*\bigg(\frac{\partial}{\partial
z^{i}}\bigg)=\frac{\partial}{\partial
z^{i}}, \]
so the $\tilde{w}_{i}$ generate $\mathcal{D}$. Moreover, by naturality of Lie brackets, we have that 
\[\Pi_*([\tilde{w}_i,\tilde{w}_j])=[\Pi_*(\tilde{w}_i),\Pi_*(\tilde{w}_j)].\]
The mapping $\Pi_*$ induces a map $\Pi_{*,k}:\T_M\otimes \Ol_{S(k)}\to\Ol_{S(k)}^l,$ given by:
\[\Pi_{*,k}(v\otimes [1]_{k+1})=\Pi_*(\tilde{v})\otimes[1]_{k+1}.\]
This map is injective when restricted to $\F$; since $[w_i,w_j]\in\F$ and $\Pi_{*,k}([w_i,w_j])=0$ we have that $[w_i,w_j]=0$.
We want now to modify the $\tilde{w}_i$'s to obtain $l$ independent commuting sections of $\F$, without changing their
equivalence class. Therefore, we look for extensions of the $w_i$'s which satisfy the thesis of the theorem, proceeding
by induction on the number of sections. If $l'=1$, we can take any extension of $w_{m+1}$ (every vector field commutes
with itself).
Suppose now the claim is true for $l'-1$ sections. Then, by the Holomorphic Frobenius theorem there exists a coordinate
chart
adapted to $S$ in which
$\tilde{w}_{m+1}=\partial/\partial z^{m+1},\ldots,\tilde{w}_{m+l'-1}=\partial/\partial z^{m+l'-1}.$
Now, since the $w_{i}$ are commuting when restricted to $S(k),$ if 
\[w_{m+l'}=[g^v]_{k+1}\frac{\partial}{\partial z^v}+[f^i]_{k+1}\frac{\partial}{\partial z^i},\]
we have that:
\begin{align*}
[0]_{k+1}=\frac{\partial [g^v]_{k+1}}{\partial z^{i}}\frac{\partial}{\partial z^v}+\frac{\partial [f^j]_{k+1}}{\partial
z^{i}}\frac{\partial}{\partial z^j}=\bigg[\frac{\partial g^v}{\partial z^{i}}\bigg]_{k+1}\frac{\partial}{\partial
z^v}+\bigg[\frac{\partial f^j}{\partial z^{i}}\bigg]_{k+1}\frac{\partial}{\partial z^j},
\end{align*}
where $i$ ranges in $m+1,\ldots,m+l'-1$.
The last equality tells us that:
\[\frac{\partial g^v}{\partial z^{i}}=z^{r_1}\cdots z^{r_{k+1}} h^v_{r_1,\cdots,r_{k+1},i}\,,\quad\frac{\partial
f^j}{\partial
z^{i}}=z^{r_1}\cdots z^{r_{k+1}} h^j_{r_1,\cdots,r_{k+1},i}. \]
We have to find $\tilde{g^v}$, $\tilde{f}^j$ representatives for the classes $[g^v]_{k+1}$,$[f^j]_{k+1}$ such that 
\[0=\frac{\partial \tilde{g}^v}{\partial z^{i}}\frac{\partial}{\partial z^v}+\frac{\partial \tilde{f}^j}{\partial
z^{i}}\frac{\partial}{\partial z^j}.\]
We do that for one of the $g^v$'s, the method applies to all the other coefficients.
Now, $\tilde{g}^v=g^v+z^{r_1}\cdots z^{r_{k+1}} \tilde{h}_{r_1,\ldots,r_{k+1}}$, so
\begin{align*}
\frac{\partial \tilde{g}^v}{\partial z^{i}}&=\frac{\partial g^v}{\partial z^{i}}+z^{r_1}\cdots
z^{r_{k+1}}\frac{\partial \tilde{h}^v_{r_1,\ldots,r_{k+1}}}{\partial z^{i}}\\
&=z^{r_1}\cdots z^{r_{k+1}}
h^v_{r_1,\ldots,r_{k+1},i}+z^{r_1}\cdots z^{r_{k+1}}\frac{\partial \tilde{h}^v_{r_1,\ldots,r_{k+1}}}{\partial z^{i}}.
\end{align*}
Therefore, the problem reduces to find a primitive $\tilde{h}^v_{r_1,\ldots,r_{k+1}}$ for the $1$-form 
\[\omega:=-h^v_{r_1,\ldots,r_{k+1},i} dz^i,\]
where the other coordinates are considered as parameters. 
If we denote by $\partial$ the holomorphic differential and supposing, without loss of generality, that $U$ is simply connected and centered at $x\in S$ (i.e. $\phi(x)=0$)
we have, by the Holomorphic Poincar\'{e} Lemma, that this primitive exists if and only if $\omega$ is closed. Therefore
we need to check that the mixed partial derivatives coincide:
\[z^{r_1}\cdots z^{r_{k+1}}\frac{\partial h^v_{r_1,\ldots,r_{k+1},i}}{\partial z^j}=\frac{\partial^2 g^v}{\partial
z^j\partial z^i}=\frac{\partial^2
g^v}{\partial z^i\partial z^j}=z^{r_1}\cdots z^{r_{k+1}}\frac{\partial h^v_{r_1,\ldots,r_{k+1},j}}{\partial z^i}.\]
Then, the primitive exists and is defined in $U$ by:
\[\tilde{h}^v_{r_1,\ldots,r_{k+1}}(z^1,\ldots, z^n)=\int_{\gamma}-h^v_{r_1,\ldots,r_{k+1},i} dz^i,\]
where $\gamma$ is a curve such that $\gamma(1)=(z^1,\ldots,z^n)$ and $\gamma(0)=0$.
\end{proof}
As a simple consequence of the Lemma, we have the Frobenius Theorem for $k$-th infinitesimal neighborhoods.
\begin{Corollary}[Frobenius Theorem for $k$-th infinitesimal neighborhoods]\label{Frobenius}$\,$\\
Suppose $S$ is a non singular complex submanifold of codimension $m$ in a complex manifold $M$ of dimension $n$ and
suppose we have a foliation $\F$ of $S(k)$ of rank $l$. Then there exists an atlas $\{U_{\alpha},\phi_{\alpha}\} $adapted to $S$ such that if $U_{\alpha}\cap U_{\beta}\cap S\neq \emptyset$ then:
\begin{equation}\label{secondorder}
\bigg[\frac{\partial z_{\alpha}^{t}}{\partial z_{\beta}^{i}}\bigg]_{k+1}=0,
\end{equation}
for $t=1,\ldots,m,m+l+1,\ldots,n$ and $i=1,\ldots,l$ on $U_{\alpha}\cap U_{\beta}$.
\end{Corollary}
\begin{proof}
We take an atlas adapted to $S$ and extensions $\tilde{w}_{i,\alpha}$ as given by Lemma \ref{Lemma:Commute}. 
By the usual Holomorphic Frobenius theorem, there exist a coordinate system (modulo shrinking) on $U_{\alpha}$ such
that 
\[\tilde{w}_{m+1,\alpha}=\frac{\partial}{\partial z^{m+1}_{\alpha}},\ldots,\tilde{w}_{m+l,\alpha}=\frac{\partial}{\partial
z^{m+l}_{\alpha}}.\] 
We take such coordinate systems. Since we are dealing with a foliation of $S(k)$, we know that if
$U_{\alpha}\cap U_{\beta}\cap S\neq \emptyset$ and $\F$ is generated on each $U_{\alpha}\cap S$ by
$w_{1,\alpha},\ldots,w_{l,\alpha}$ we have that
$w_{i,\alpha}=[(c_{\alpha\beta})_{i}^j]_{k+1} w_{j,\beta}$. Hence:
\begin{align*}
\bigg[\frac{\partial z^{t}_{\alpha}}{\partial z^{i}_{\beta}}\bigg]_{k+1}&=\tilde{w}_{i,\beta}\otimes
[1]_{k+1}(z^{t}_{\alpha})=w_{i,\beta}(z^{t}_{\alpha})=[c_{i}^j]_{k+1} w_{j,\alpha}(z^{t}_{\alpha})\\
=& [c_{i}^j]_{k+1}\tilde{w}_{j,\alpha}\otimes [1]_{k+1}(z^{t}_{\alpha})=[c_{i}^j]_{k+1}\bigg[\frac{\partial
z^{t}_{\alpha}}{\partial z^{j}_{\alpha}}\bigg]_{k+1}=[c_{i}^j\delta_j^t]_{k+1}=[0]_{k+1}.
\end{align*}
\end{proof}
\begin{Remark}
It is easily seen that the existence of an atlas satisfying \eqref{secondorder} implies the existence of a foliation of $\T_{S(k)}$, generated on each chart $U_{\alpha}$ intersecting $S$ by $\{\partial/\partial z^{m+1}_{\alpha},\ldots,\partial/\partial z^{m+l}_{\alpha}\}$.
\end{Remark}
\begin{Definition}\label{CondStar}
We say that a foliation $\F$ of $S$ \textbf{extends to the $k$-th infinitesimal neighborhood}, if there exists an atlas adapted to $S$ and $\F$ such that:
\[\bigg[\frac{\partial z_{\beta}^{t}}{\partial z_{\alpha}^{i}}\bigg]_{k+1}=0,\]
for $t=1,\ldots,m,m+l+1,\ldots,n$ and $i=1,\ldots,l$ on $U_{\alpha}\cap U_{\beta}$.\\
In the special case $\F=\T_S$ we say that $S$ has \textbf{$k$-th order extendable tangent bundle}.
\end{Definition}
\begin{Remark}
Let $M$ be a complex manifold, $\F$ a regular foliation of $M$. Every leaf of $\F$ has $k$-th order extendable tangent
bundle for every $k$.
\end{Remark}
\begin{Remark}
For a submanifold $S$ having first order extendable tangent bundle is a strong topological condition. As a matter of
fact, as we will see in section \ref{IndFolia} of this paper, this implies the vanishing of many of the characteristic
classes of the normal bundle of $S$.
\end{Remark}
\begin{Remark}
If a submanifold $S$ has first order extendable tangent bundle, it is likely that
every
foliation on $S$ extends to a foliation of the first infinitesimal neighborhood. A result in this direction can be found
in Corollary \ref{FirstOrder}.
\end{Remark}

\section{The Atiyah sheaf for the Variation Action}\label{SecAtiyah}
The Atiyah sheaf is an important geometrical object defined in \cite{Atiyah}. In that paper, it is proved that the
existence of a holomorphic connection for a vector bundle $E$ is equivalent to the splitting of the following
sequence:
\begin{equation}\label{seqAti}
0\rightarrow \Hom(E,E)\rightarrow \mathcal{A}_{E}\rightarrow TM\rightarrow 0,
\end{equation}
where $\mathcal{A}_{E}$ is the Atiyah sheaf of $E$.
In \cite{ABTHolMap} it was proved that the obstruction to the existence of partial holomorphic connections for a vector bundle $E$ along a subbundle $\F$ is equivalent to the splitting of the following sequence:
\begin{equation}\label{seqAtipart}
\xymatrix{0\ar[r]&\Hom(E,E)\ar[r]&\mathcal{A}_{E,\F}\ar[r]&\F\ar[r]&0 }. 
\end{equation}
\begin{Definition}
Let $S$ be a not necessarily closed complex submanifold of $M$, $\F$ a foliation of $S$. Let
$\T_{M,S(1)}:=\T_M\otimes_{\Ol_M}\Ol_{S(1)}$ and $\T_{M,S}:=\T_M\otimes_{\Ol_M}\Ol_{S}$; if
$\theta_1:\Ol_{S(1)}\to\Ol_S$ is the canonical projection, we denote by
$\Theta_1$ the map $\id\otimes \theta_1:\T_{M,S(1)}\to\T_{M,S}.$
We see $\F$ as a subsheaf of $\T_{M,S}$, and we define the \textbf{normal bundle to the foliation in the ambient tangent bundle} as the quotient of $\T_{M,S}$ by $\F$ and we will denote it by $\N_{\F,M}$. 
Let $\T_{M,S(1)}^{\F}:=\ker(\pr\circ\Theta_1)$, where $\pr$ is the quotient map in the short
exact sequence:
\begin{equation}\label{SesImp}
\xymatrix{
& & \T_S \ar[d]& &\\
0\ar[r]&\F \ar[ur] \ar[r]^i&\T_{M,S}\ar[r]^{\pr}&\N_{\F,M}\ar[r]&0.
} 
\end{equation}
\end{Definition}
\begin{Remark}
In our case, we have to replace $E$ in \eqref{seqAti} with $\N_{\F,M}$; the computation of the
obstruction to the splitting of this
sequence is a straightforward application of the procedure in \cite{Atiyah} and therefore we omit
it.
In an atlas adapted to $S$ and $\F$, in \u{C}ech-deRham cohomology the class is represented by the
cocycle
\[\bigg\{ U_{\alpha\beta},-\frac{\partial z_{\alpha}^{t'}}{\partial z^t_{\beta}}\frac{\partial^2
z^t_{\beta}}{\partial
z^i_{\alpha}\partial z^w_{\alpha}}\bigg|_S dz^i_{\alpha}\otimes\omega^w_{\beta}\otimes
\partial_{t',\beta}\bigg\},\]
where $\{\partial_{t,\alpha}\}$ is the quotient frame for $N_{F,M}$ in $U_{\alpha}$ and
$\omega^t_{\alpha}$ is the dual
frame for $\N_{\F,M}$ on $U_{\alpha}$.
\end{Remark}
As in \cite{ABTHolMapFol}, we will define a more concrete realization of the Atiyah sheaf for the
sheaf
$\mathcal{N}_{\F,M}$. We shall prove that the splitting of the Atiyah sequence for $\N_{\F,M}$ is equivalent to
the fact that the foliation $\F$ extends to the first infinitesimal neighborhood.
\begin{Remark} By definition $\Theta_1(\T_{M,S(1)}^{\F})$ is contained in the kernel of $\pr$, so, by exactness
of sequence \eqref{SesImp}, it is contained in the image of $\F$. Moreover, for each $v\in\F$, at least locally, the element $\tilde{v}\otimes [1]_2$
belongs to $\T_{M,S(1)}^{\F}$ and is projected by $\Theta_1$ to $i(v)$.
So, $\Theta_1(\T_{M,S(1)}^{\F})=i(\F)$.
\end{Remark}
\begin{Remark}\label{incoord}
Suppose we have a coordinate system adapted to $S$ and $\F$ (Definition \ref{Def:AdaptedFS}). Then $v$ belongs to
$\T_{M,S(1)}^{\F}$ if and only if 
$v=[a^k]_2\partial/\partial z^k,$ with $[a^t]_1=0$, where $t=1,\ldots,m,m+l+1,\ldots,n$.
Analogously $v$ belongs to $\mathcal{I}_S\T_{M,S(1)}^{\F}$ if and only if $v=[a^i]_2\partial/\partial z^i$, where
$a^i\in\mathcal{I}_S$ for $i=m+1,\ldots,m+l$. 
\end{Remark}
\begin{Lemma}\label{Algebroide}
Let $S$ be a complex submanifold of a complex manifold $M$ and $\F$ a foliation of $S$. Then
\begin{enumerate}
\item every $v$ in $\T_{M,S(1)}^{\F}$ induces a derivation $g\mapsto v(g)$ of $\Ol_{S(1)}$;
\item there exist a natural $\mathbb{C}$-linear map \label{commutatore}
$\{\ \cdot\ ,\ \cdot\ \}:\T_{M,S(1)}^{\F}\otimes\T_{M,S(1)}^{\F}\to\T_{M,S(1)}^{\F}$ such that
\begin{enumerate}
\item $\{u,v\}=-\{v,u\},$
\item $\{u,\{v,w\}\}+\{v,\{w,u\}\}+\{w,\{u,v\}\}=0,$
\item $\{gu,v\}=g\{u,v\}- v(g) u,$, for all $g\in\Ol_{S(1)}$
\item $\Theta_1(\{u,v\})=[\Theta_1(u),\Theta_1(v)].$
\end{enumerate}  
\end{enumerate}
\end{Lemma}
\begin{proof}
\begin{enumerate}
\item Let $(U;z^1,\ldots,z^n)$ be a coordinate chart adapted to $S$ and $\F$. An element
$v=[a^k]_2\frac{\partial}{\partial z^k}\in \T_{M,S(1)}$ belongs to $\T_{M,S(1)}^{\F}$ if and only if $[a^{t}]_1=0.$
Remembering Remark \ref{StruttLog} we see that $v$ belongs to $\T_{S(1)}$ and Lemma \ref{Lemma:Logaritmici} gives the
assertion. 
\item We define $\{\ \cdot\ ,\ \cdot\ \}$ by setting 
\[\{u,v\}(f)=u(v(f))-v(u(f)),\]
for every $f\in \Ol_{S(1)}$. Please note that, since $u$ and $v$ belong to $\T_{S(1)}$ this bracket coincides with the
bracket defined on $\T_{S(1)}$; the first three properties are proved exactly as for the usual bracket of vector
fields, while the fourth follows from a simple computation in coordinates. Suppose $(U;z^1,\ldots,z^n)$ is a
coordinate chart adapted to $S$ and $\F$, $u=[a^k]_2\frac{\partial}{\partial z^k},\,v=[b^k]_2\frac{\partial}{\partial
z^k}$ with $[a^t]_1=0$ and $[b^t]_1=0$. First of all we compute the Lie brackets on $\T_{M,S(1)}^{\F}$ in coordinates: 
\begin{align*}
\{u,v\}&=\bigg[a^h\frac{\partial b^k}{\partial z^h}-b^h\frac{\partial a^k}{\partial z^h}\bigg]_2\frac{\partial}{\partial
z^k}\\
&=\bigg[a^t\frac{\partial b^u}{\partial z^t}+a^i\frac{\partial b^u}{\partial
z^i}-b^t\frac{\partial a^u}{\partial z^t}-b^i\frac{\partial a^u}{\partial
z^i}\bigg]_2\frac{\partial}{\partial z^u}\\
&+\bigg[a^t\frac{\partial b^j}{\partial z^t}-b^t\frac{\partial a^j}{\partial z^t}\bigg]_2\frac{\partial}{\partial z^j}
+\bigg[a^i\frac{\partial b^j}{\partial z^i}-b^i\frac{\partial a^j}{\partial z^i}\bigg]_2\frac{\partial}{\partial z^j}.
\end{align*}
Please note that the coefficients in the first two summands of the last expression all belong to $\mathcal{I}_S/\mathcal{I}_S^2$. Therefore:
\begin{align*}
\Theta_1(\{u,v\})&=\bigg[a^i\frac{\partial b^j}{\partial z^i}-b^i\frac{\partial a^j}{\partial
z^i}\bigg]_1\frac{\partial}{\partial
z^j}=[\Theta_1(u),\Theta_1(v)].
\end{align*}
\end{enumerate}
\end{proof}
\begin{Remark}
In general, given two vector fields $u,v$ in $\T_{M,S(1)}$, we can define a bracket as:
$[u,v](f)=u(v(f))-v(u(f)),$
for $f\in \Ol_{S(1)}$.
Please note that this bracket is not a well defined section of $\T_{M,S(1)}$ but only of $\T_{M,S}$. In other words
$[u(v(f))-v(u(f))]_2$ is not well defined, while $[u(v(f))-v(u(f))]_1$ is.
\end{Remark}
\begin{Lemma}\label{BenDef}
Let $S$ be an $m$-codimensional complex submanifold of a complex manifold $M$ of complex dimension $n$ and $\F$ a
foliation
of $S$. Then:
\begin{enumerate}
 \item $u\in \T_{M,S(1)}^{\F}$ is such that $\pr([u,s])=0$ for all $s\in\T_{M,S(1)}$ if and only if $u\in
\mathcal{\mathcal{I}}_S\T_{M,S(1)}^{\F};$\label{benDefConn}
\item if $u\in \mathcal{\mathcal{I}}_S\T_{M,S(1)}^{\F}$ and $v\in \T_{M,S(1)}^{\F}$ then $\{u,v\}\in
\mathcal{\mathcal{I}}_S\T_{M,S(1)}^{\F};$\label{benDef}
\item the quotient sheaf 
\[\mathcal{A}=\T_{M,S(1)}^{\F}/\mathcal{\mathcal{I}}_S\T_{M,S(1)}^{\F}\]
admits a natural structure of $\Ol_S$-locally free sheaf such that the map induced by $\Theta_1$, whose image lies in
$\F$, is an $\Ol_S$-morphism.
\end{enumerate}
\end{Lemma}
\begin{proof}
\begin{enumerate}
 \item Writing $u=[a^k]_2\frac{\partial}{\partial z^k},$ with $[a^t]_1=0$, and $s=[b^h]_2\frac{\partial}{\partial
z^h}\in \T_{M,S(1)}$, we have:
\[\pr([u,s])=\bigg[a^k\frac{\partial b^t}{\partial z^k}-b^k\frac{\partial a^t}{\partial
z^k}\bigg]_1\frac{\partial}{\partial z^t}.\]
If $u$ belongs to $\mathcal{\mathcal{I}}_S \T_{M,S(1)}^{\F}$ clearly $\pr([u,s])=0.$

Conversely, let $u$ be such that $\pr([u,s])=0$ for each $s\in \T_{M,S(1)}.$ We claim it belongs to
$\mathcal{I}_S\T_{M,S(1)}^{\F}$.
We know that $u$ belongs to $\T_{M,S(1)}^{\F}$, so that $[a^t]_1=0$. Let $s=\partial/\partial z^r$,
with $r=1,\ldots,m$. Then $[\partial
a^t/\partial z^r]_1=0$. Now, we take a representative $h_s z^s$, with $s=1,\ldots,m$, for the class
$[a^t]_1$. Computing:
\[0=\bigg[\frac{\partial a^t}{\partial z^r}\bigg]_1=\bigg[\frac{\partial h_s}{\partial z^s}z^s+h_s\delta^s_r\bigg]_1=[h_s]_1.\]
So, for each $s$, we have that $h_s$ belongs to $\mathcal{I}_S$, implying that $[a^t]_2=0$.
Fix now a $j$ in $m+1,\ldots,n$ and let $s=[z^{j}]_2\frac{\partial}{\partial z^1}$.
Then
\[0=-\bigg[z^{j}\frac{\partial a^t}{\partial z^1}\bigg]_1\frac{\partial}{\partial
z^t}+\bigg[a^k\delta_k^{j}\bigg]_1\frac{\partial}{\partial z^1}=[a^{j}]_1\frac{\partial}{\partial z^1},\]
where the last equality follows from the preceeding step, where we proved that $[a^{t}]_2=0$ and thus that
$\bigg[\frac{\partial a^{t}}{\partial z^1}\bigg]_1=0.$
So, $[a^j]_1=0$ and $u$ belongs to $\mathcal{I}_S\T_{M,S(1)}^{\F}$.
\item This follows by a direct computation in coordinates.
\item The sheaf $\T_{M,S(1)}^{\F}$ is an $\Ol_{S(1)}$-submodule of $\T_{M,S(1)}$ such that $g\cdot v$ belongs to
$\mathcal{I}_S\T_{M,S(1)}^{\F}$ for every $g\in \mathcal{I}_S/\mathcal{I}_S^2$ and $v\in \T_{M,S(1)}^{\F}$. Therefore
the $\Ol_{S(1)}$ structure induces a natural $\Ol_S$-module structure on $\mathcal{A}$.\\
Remember $\T_{M,S(1)}^{\F}$ is generated locally, in an atlas adapted to $S$ by $\partial/\partial z^j$, with
$j=m+1,\ldots,m+l$ and by $[z^r]_2\partial/\partial z^s$, with $r$ and $s$ varying in $1,\ldots,m$. 
Then, the sheaf $\mathcal{A}$ is a locally free $\Ol_S$-module freely generated by $\pi(\frac{\partial}{\partial z^j})$
and $\pi([z^s]_2\frac{\partial}{\partial z^t})$, where $\pi:\T_{M,S(1)}^{\F}\to\mathcal{A}$ is the quotient map.
Moreover, $\mathcal{I}_S\T_{M,S(1)}^{\F}$ lies in the kernel of $\Theta_1$ so $\Theta_1$ factors through a map that we
will denote again by $\Theta_1:\mathcal{A}\to \F,$ which is clearly an $\Ol_S$-morphism.
\end{enumerate}
\end{proof}
\begin{Definition}
Let $S$ be a complex submanifold of a complex manifold $M$ and let $\F$ be a foliation of $S$.
The \textbf{Atiyah sheaf of $\F$} is the locally free $\Ol_S$-module
\[\mathcal{A}=\T_{M,S(1)}^{\F}/\mathcal{I}_S\T_{M,S(1)}^{\F}.\]
\end{Definition}
\begin{Theorem}\label{TheoAtiyahSplitting}
Let $S$ be a codimension $m$ submanifold of a complex manifold $M$ of dimension $n$ and let $\F$ be a foliation of $S$.
Then there exists a natural exact sequence of locally free $\Ol_S$-modules
\[\xymatrix{ 0\ar[r] &\Hom(\mathcal{N}_S,\mathcal{N}_{\F,M})\ar[r]&\mathcal{A}\ar[r]^{\Theta_1}&\F\ar[r]&0 }.\]
whose splitting is equivalent to the splitting of the sequence \eqref{seqAtipart} taking as $E$
$\N_{\F,M}$.
\end{Theorem}
\begin{proof}
We work in a chart adapted to $S$ and $\F$. The kernel of $\Theta_1$ is locally freely generated by the images under
$\pi:\T_{M,S(1)}^{\F}\to\mathcal{A}$ of $[z^s_{\alpha}]_2\frac{\partial}{\partial z^t_{\alpha}}$. We would like to
understand how the generators behave under change of coordinates, to see if $\ker(\Theta_1)$ is isomorphic to any known sheaf of sections of a known vector bundle.
We compute the coordinate change maps:
\begin{align}
\pi\bigg([z^s_{\alpha}]_2\frac{\partial}{\partial z^{t}}\bigg)&=\pi\bigg([z^s_{\alpha}]_2\bigg[\frac{\partial
z^k_{\beta}}{\partial z^{t}_{\alpha}}\bigg]_2\frac{\partial}{\partial
z^k_{\beta}}\bigg)=\pi\bigg([z^s_{\alpha}]_2\bigg[\frac{\partial z^w_{\beta}}{\partial
z^{t}_{\alpha}}\bigg]_2\frac{\partial}{\partial z^w_{\beta}}\bigg)\label{EqThSplitAtiyahPrimaEq} \\ 
&=\pi\bigg(\bigg[\frac{\partial z^s_{\alpha}}{\partial z^{s_1}_{\beta}}\bigg]_2[z^{s_1}_{\beta}]_2\bigg[\frac{\partial
z^w_{\beta}}{\partial z^{t}_{\alpha}}\bigg]_2\frac{\partial}{\partial
z^w_{\beta}}\bigg)\nonumber\\&=\bigg[\frac{\partial
z^s_{\alpha}}{\partial z^{s_1}_{\beta}}\frac{\partial z^w_{\beta}}{\partial
z^{t}_{\alpha}}\bigg]_1\pi\bigg([z^{s_1}_{\beta}]_2\frac{\partial}{\partial
z^w_{\beta}}\bigg),\label{EqThSplitAtiyahTerzaEq}
\end{align}
where the last equality in \eqref{EqThSplitAtiyahPrimaEq} comes from the quotient map and the one in
\eqref{EqThSplitAtiyahTerzaEq} comes from the newly acquired structure of $\Ol_S$-module. As a consequence, the kernel
of
$\Theta_1$ is isomorphic to $\Hom(\mathcal{N}_S,\mathcal{N}_{\F,M})$.
Now, if we define local splittings of the sequence by setting 
\[\sigma_{\alpha}\bigg(\frac{\partial}{\partial z^j_{\alpha}}\bigg)=\pi\bigg(\frac{\partial}{\partial
z^j_{\alpha}}\bigg),\]
for $j=m+1,\ldots,m+l$,
and extending by $\Ol_S$-linearity, we can compute the obstruction to find a splitting of the sequence:
\begin{align}
(\sigma_{\beta}-\sigma_{\alpha})\bigg(\frac{\partial}{\partial
z^j_{\beta}}\bigg)&=\sigma_{\beta}\bigg(\frac{\partial}{\partial
z^j_{\beta}}\bigg)-\sigma_{\alpha}\bigg(\bigg[\frac{\partial z^i_{\alpha}}{\partial
z^j_{\beta}}\bigg]_1\frac{\partial}{\partial z^i_{\alpha}}\bigg)\nonumber\\
&=\sigma_{\beta}\bigg(\frac{\partial}{\partial z^j_{\beta}}\bigg)-\bigg[\frac{\partial
z^i_{\alpha}}{\partial z^j_{\beta}}\bigg]_1\sigma_{\alpha}\bigg(\frac{\partial}{\partial
z^i_{\alpha}}\bigg)\nonumber\\
&=\pi\bigg(\frac{\partial}{\partial z^j_{\beta}}\bigg)-\bigg[\frac{\partial z^i_{\alpha}}{\partial
z^j_{\beta}}\bigg]_1\pi\bigg(\frac{\partial}{\partial z^i_{\alpha}}\bigg)\nonumber\\
&=\pi\bigg(\bigg[\frac{\partial z^t_{\alpha}}{\partial z^j_{\beta}}\bigg]_2\frac{\partial}{\partial z^t_{\alpha}}\bigg)
=\pi\bigg(\bigg[\frac{\partial^2 z^t_{\alpha}}{\partial z^r_{\beta}\partial
z^j_{\beta}}z^r_{\beta}\bigg]_2\frac{\partial}{\partial z^t_{\alpha}}\bigg)\nonumber\\
&=\bigg[\frac{\partial^2 z^t_{\alpha}}{\partial z^r_{\beta}\partial z^j_{\beta}}\frac{\partial z^r_{\beta}}{\partial
z^s_{\alpha}}\bigg]_1\pi\bigg([z^s_{\alpha}]_2\frac{\partial}{\partial
z^t_{\alpha}}\bigg)\label{eq:ObstructionNormal}.
\end{align}
Please remark that, since $\partial z^t_{\alpha}/\partial z^j_{\beta}$ lies in the ideal
$\mathcal{I}_S$ for
$t=1,\ldots,m,m+l+1,\ldots,n$ and $j=m+1,\ldots,m+l$ it follows that
\[\frac{\partial^2 z^t_{\alpha}}{\partial z^p_{\beta}\partial z^j_{\beta}}\in \mathcal{I}_S\]
for $t=1,\ldots,m,m+l+1,\ldots,n$, $j=m+1,\ldots,m+l$ and $p=m+1,\ldots,n$.
Therefore we have that
\[\bigg[\frac{\partial^2 z^t_{\alpha}}{\partial z^w_{\beta}\partial z^j_{\beta}}\bigg]_1=[0]_1\]
for $t,w=1,\ldots,m,m+l+1,\ldots,n$ and $j=m+1,\ldots,m+l$ if and only if
\[\bigg[\frac{\partial^2 z^t_{\alpha}}{\partial z^r_{\beta}\partial z^j_{\beta}}\bigg]_1=[0]_1\]
for $t,w=1,\ldots,m,m+l+1,\ldots,n$, $j=m+1,\ldots,m+l$ and $r=1,\ldots,m$.
Hence, class \eqref{eq:ObstructionNormal} vanishes if and only if \eqref{seqAtipart} splits.
\end{proof}
It is easily noted that in the case $\F$ is the tangent bundle to $S$ the Atiyah sheaf of $\F$ is nothing else than the
Atiyah sheaf of $S$, defined in \cite{ABTHolMapFol}.

\begin{Definition}
Let $\F$ be a sheaf of $\Ol_S$-modules over a complex manifold $S$, equipped with a $\Ol_S$-morphism $X:\F\to\T_S$. We
say that $\F$ is a \textbf{Lie algebroid of anchor $X$} if there is a $\mathbb{C}$-bilinear map $\{\ \cdot\ ,\,\cdot\
\}:\F\oplus\F\to\F$ such that:
\begin{enumerate}
\item $\{v,u\}=-\{u,v\}$;
\item $\{u,\{v,w\}\}+\{v,\{w,u\}\}+\{w,\{u,v\}\}=0$;
\item $\{g\cdot u, v\}=g\cdot \{u,v\}-X(v)(g)\cdot u$ for all $g\in\Ol_S$ and $u,v\in\F.$
\end{enumerate}
\end{Definition}
\begin{Definition}
Let $\mathcal{E}$ and $\mathcal{F}$ be locally free sheaves of $\Ol_S$-modules over a complex manifold $S$. Given a section $X\in H^0(S,\T_S\otimes\F^*)$, a \textbf{holomorphic $X$-connection on $\mathcal{E}$} is a $\mathbb{C}$-linear map $\tilde{X}:\mathcal{E}\to\F^*\otimes\mathcal{E}$ such that:
\[\tilde{X}(g\cdot s)=X^*(dg)\otimes s+g\tilde{X}(s),\]
for each $g\in\Ol_S$ and $s\in\mathcal{E}$, where $X^*$ is the dual map of $X$. The notation $\tilde{X}_v(s)$  is equivalent to $\tilde{X}(s)(v)$.

If $\F$ is a Lie algebroid of anchor $X$ we define the \textbf{curvature} of $\tilde{X}$ to be:
\[R_{u,v}(s)=\tilde{X}_u\circ\tilde{X}_v(s)-\tilde{X}_v\circ\tilde{X}_u(s)-\tilde{X}_{\{u,v\}}(s).\]
We say that $\tilde{X}$ is \textbf{flat} if $R\equiv 0$.
\end{Definition}

\begin{Proposition}\label{PropHolConnection}
Let $S$ be a complex submanifold of a complex manifold $M$ and $\F$ a holomorphic foliation of $S$. Then:
\begin{enumerate}
 \item the Atiyah sheaf of $\F$ has a natural structure of Lie algebroid of anchor $\Theta_1$ such that
 \[\Theta_1\{q_1,q_2\}=[\Theta_1(q_1),\Theta_1(q_2)]\]
for all $q_1,q_2\in\mathcal{A};$
\item there is a natural holomorphic $\Theta_1$-connection $\tilde{X}:\mathcal{N}_{\F,M}\to\mathcal{A}^*\otimes
\mathcal{N}_{\F,M}$ on $\N_{\F,M}$ given by
\[\tilde{X}_q(s)=\pr([v,\tilde{s}])\]
for all $q\in\mathcal{A}$ and $s\in N_{\F,M}$, where $v\in\T_{M,S(1)}^{\F}$ and $\tilde{s}\in\T_{M,S(1)}$ are such that
$\pi(v)=q$ and $\pr\circ\Theta_1(\tilde{s})=s;$
\item this holomorphic $\Theta_1$-connection is flat.
\end{enumerate}
\end{Proposition}
\begin{proof}
\begin{enumerate}
 \item We set
\[\{q_1,q_2\}=\pi(\{v_1,v_2\}),\]
where $v_i\in\T_{M,S(1)}^{\F}$ are such that $q_i=\pi(v_i),$ for $i=1,2$.
This is well defined: if $q_1=0,$ then $v_1$ is in $\mathcal{I}_S\T_{M,S(1)}^{\F}$ and then, by \ref{benDef} of Lemma
\ref{BenDef} we have
that 
$\{q_1,q_2\}=0$. The other properties follow directly from Lemma \ref{Algebroide}.
\item We check the connection is well defined. Suppose now $q=0$; this means that
$v\in\mathcal{I}_S\T_{M,S(1)}^{\F}$; then, by Lemma \ref{BenDef}.\ref{benDefConn}, we have that
$\pr([v,\tilde{s}])=0$, for every $\tilde{s}\in\T_{M,S(1)}$. Now, if
$\pr\circ\Theta_1(\tilde{s})=0,$ we have that $\tilde{s}\in\T_{M,S(1)}^{\F}$, so $\{v,\tilde{s}\}$
is in $\T_{M,S(1)}^{\F}$, which implies that $\tilde{X}_q(s)=0$. \\
We check now it is a $\Theta_1$-connection. It is $\Ol_S$-linear in the first entry since:
\begin{align*}
\tilde{X}_{[f]_1\cdot q}(s)=\pr([[f]_2 v,\tilde{s}])=\pr([f]_1 [v,\tilde{s}]-\tilde{s}([f]_2)\Theta_1(v))=[f]_1
\tilde{X}_q(v),
\end{align*}
 where the last equality comes from the fact that $v$ belongs to $\T_{M,S(1)}^{\F},$ which is the kernel of
$\pr\circ\Theta_1$.
We check the $\Theta_1$-Leibniz rule for the second entry:
\begin{align*}
\tilde{X}_q([f]_1 s)&=\pr([v,[f]_2\tilde{s}])=\pr([f]_1[v,\tilde{s}]+v([f]_2)\cdot\Theta_1(\tilde{s}))\\
&=[f]_1\tilde{X}_q(s)+\Theta_1(q)([f]_1)\cdot s,
\end{align*}
where the last equality comes from the equality:
\[[v([f]_2)]_1=\Theta_1(v)([f]_1)=\Theta_1(\pi(v))([f]_1),\]
for every $[f]_2\in \Ol_{S(1)}$ and for every $v\in\T_{M,S(1)}^{\F}$. Thus, $\tilde{X}$ is a holomorphic
$\Theta_1$-connection.
\item We compute the curvature:
\begin{align*}
R_{q_1,q_2}(s)&=\tilde{X}_{q_1}\circ\tilde{X}_{q_2}(s)-\tilde{X}_{q_2}\circ\tilde{X}_{q_1}(s)-\tilde{X}_{\{q_1,q_2\}}(s)\\
&=\pr([u , \widetilde{\pr([v,\tilde{s}])}])-\pr([v,\widetilde{\pr([u,\tilde{s}])}])-\pr([[u,v],\tilde{s}]).
\end{align*}
As we proved before, the connection does not depend on the extension chosen for the second entry, so we can rewrite the expression as:
\[\pr([u,[v,\tilde{s}]])-\pr([v,[u,\tilde{s}]])-\pr([[u,v],\tilde{s}]).\]
Computing in coordinates, it follows from the usual Jacobi identity for vector fields that it is identically $0$.
\end{enumerate}
\end{proof}
\begin{Definition}
Let $S$ be a complex submanifold of a complex manifold $M$ and $\F$ a foliation of $S$. The holomorphic $\Theta_1$-
connection $\tilde{X}:\mathcal{N}_{\F,M}\to\mathcal{A}^*\otimes \mathcal{N}_{\F,M}$ just introduced is called
the \textbf{universal holomorphic connection} on $\N_{\F,M}$.
\end{Definition}
\begin{Corollary}\label{CorollaryHolomorphicAction}
Let $S$ be a submanifold of a complex manifold $M$ and suppose there exists a foliation $\F$ of the
first infinitesimal neighborhood of $S$. Then, there exists a flat partial holomorphic connection
$(\delta,\F)$ on $\N_{\F,M}$ along $\F$. 
\end{Corollary}
\begin{proof}
We want to define now the splitting map between $\F|_S$ and $\mathcal{A}$; this is
really simple since in an atlas adapted to $S$ and $\F$ each of $[1]_2\otimes \partial/\partial
z^i_{\alpha}$ belongs to $\T_{M,S(1)}^{\F}$.
Therefore we define $\psi:\F|_S\to\mathcal{A}$ as
\[\psi: \frac{\partial}{\partial z^i_{\alpha}}\mapsto\pi\bigg([1]_2\otimes \frac{\partial}{\partial
z^i_{\alpha}}\bigg),\]
for each $i=m+1,\ldots,m+l$, where $\pi$ is the map from $\T_{M,S(1)}^{\F}$ to
$\mathcal{A}$.
We compute now the explicit form of the induced partial holomorphic connection.
Indeed, let $v$ belong to $\F$ and $s$ belong to $\N_{\F,M}$; since $\psi(v)$ belongs to
$\T_{M,S(1)}^{\F}$, if we take a lift $\tilde{s}$ of $s$ to $\T_{M,S(1)}^{\F}$, i.e.
$\pr\circ\Theta_1(\tilde{s})=s$, we have that the partial holomorphic connection $(\delta,\F)$ along
$\F$ induced by the universal holomorphic connection for $\N_{\F,M}$  is given by:
\[\delta_v(s)=\tilde{X}_{\psi(v)}(s)=\pr([\psi(v),\tilde{s}]).\]
We prove now this partial holomorphic connection is flat; indeed
\[\delta_u(\delta_v(s))-\delta_v(\delta_u(s))-\delta_{[u,v]}((s))=\pr([\tilde{u},[\tilde{v},\tilde{s
}]] - [\tilde{v},[\tilde{u},\tilde{s}]] - [[\tilde{u},\tilde{v}],\tilde{s}])=0,\]
by the Jacobi identity.
\end{proof}

\section{Splittings and foliations of the $1$-st infinitesimal neighborhood}\label{Sec2Split}
In this section we will deal with a stronger version of splitting (Definition \ref{split}). The main idea is that,
given a splitting of a submanifold, there exist maps which permit us to ``project'' vector fields transversal to the
first infinitesimal neighborhood into vector fields which are tangential to the first infinitesimal neighborhood. We
cite the following proposition which stems from a result of commutative algebra \cite[Proposition
16.2]{EisenbudCommAlg} which ties the splitting of the conormal sequence
\begin{equation*}
\xymatrix{0\ar[r]& \mathcal{I}_S / \mathcal{I}_S^2\ar[r]^{d_2}&\Omega_{M}\otimes \Ol_S\ar[r]^p
&\Omega_S\ar[r] & 0}
\end{equation*}
with the splitting of the following short exact sequences:
\begin{equation*}
\xymatrix{0\ar[r]& \mathcal{I}_S / \mathcal{I}_S^2\ar[r]^{i_1}&\Ol_{S(1)}\ar[r]^{\theta_1}
&\Ol_S\ar[r] & 0,}\\
\end{equation*}
and the normal sequence
\begin{equation*}
\xymatrix{0\ar[r]& \T_S \ar[r]^{\iota}&\T_{M,S}\ar[r]^{p_2} &\N_S\ar[r] & 0.}
\end{equation*}
\begin{Proposition}[\cite{ABTHolMapFol} Prop. 2.7]\label{Prop:Splitting}
Let $S$ be a reduced, globally irreducible subvariety of a complex manifold $M$. Then, there
exists is a $1-1$ correspondence among the following classes of morphisms:
\begin{itemize}
 \item morphisms $\sigma:\Omega_S\to\Omega_{M}\otimes \Ol_S$ of sheaves of $\Ol_S$-modules such that $p\circ\sigma=id$;
 \item morphisms $\tau:\Omega_{M,S}\to \mathcal{I}_S/\mathcal{I}_S^2$ of sheaves of $\Ol_S$-modules such that $\tau\circ
d_2=id$;
 \item $\theta_1$-derivations $\tilde{\rho}:\Ol_{S(1)}\to \mathcal{I}_S/\mathcal{I}_S^2$ such that $\tilde{\rho}\circ
i_1=id$, where $i_1$ is the canonical inclusion;
 \item morphisms $\rho:\Ol_S\to \Ol_{S(1)}$ of sheaves of rings such that $\theta_1\circ\rho=id$.
\end{itemize}
Moreover, if any of the former classes is not empty, then there is a $1-1$ correspondence with the following classes of
morphisms
\begin{itemize}
 \item morphisms $\tau^*:\N_S\to \T_{M,S}$ of sheaves of $\Ol_S$-modules such that $p_2\circ\tau^*=id$;
 \item morphisms $\sigma^*:\T_{M,S}\to\T_S$ of sheaves of $\Ol_S$-modules such that $\iota\circ\sigma^*$, where
$\iota:\T_S\to\T_{M,S}$ is the canonical inclusion.
\end{itemize}
\end{Proposition}
\begin{Definition}\label{split}
Let $S$ be a reduced, globally irreducible subvariety of a complex manifold $M$. We say that $S$ \textbf{splits} in $M$
if there
exists a morphism of sheaves $\sigma: \Omega_S\to\Omega_{M,S}$ such that $p\circ\sigma= \id$ where
$p:\Omega_{M,S}\to\Omega_S$ is the canonical projection. 
\end{Definition}
\begin{Remark} In \cite{ABTHolMapFol} it is proved that a submanifold splits if and only if there exists an atlas
adapted to $S$ such that:
\[\bigg[\frac{\partial z^p_{\beta}}{\partial z^r_{\alpha}}\bigg]_1=0.\]
\end{Remark}
A natural generalization of the concept of splitting is the notion of $k$-splitting, developed in
\cite{ABTHolMapFol} and \cite{ABTEmbeddings}. We will use extensively the notion of $2$-splitting.
\begin{Definition}
Let $S$ be a submanifold of a complex manifold $M$. We shall say that $S$ \textbf{$k$-splits} into $M$ if and only if
there is
an infinitesimal retraction of $S(k)$ onto $S$, that is if there is a $k$-th order lifting
$\rho:\Ol_S\to\Ol_M/\mathcal{I}_S^{k+1}$ or in still other words, if the exact sequence
\begin{equation}\label{k-splitting}
0\to\mathcal{I}_S/\mathcal{I}_S^{k+1}\hookrightarrow\Ol_M/\mathcal{I}_S^{k+1}\to\Ol_M/\mathcal{I}
_S\to 0 
\end{equation}
splits as a sequence of sheaves of rings.
\end{Definition}
\begin{Remark}
Please note that, in case the sequence above splits, the map
$\tilde{\rho}:\Ol_M/\mathcal{I}_S^{k+1}\to\mathcal{I}_S/\mathcal{I}_S^{k+1}$ is a
$\theta_k$ derivation, i.e.,
\[\tilde{\rho}([fg]_{k+1})=\theta_k([f]_{k+1})\tilde{\rho}([g]_{k+1})+\theta_k([g]_{k+1})\tilde{\rho}([f]_{k+1}).\] 
The sheaf $\mathcal{I}_S/\mathcal{I}_S^{k+1}$ has  a natural structure of $\Ol_{S(1)}$-module: the
multiplications given by 
$[f]_{k}[h]_{k+1}=[fh]_{k+1}$ is well defined; indeed, let $[\tilde{f}_1]_{k+1}$ and
$[\tilde{f}_2]_{k+1}$ be two representatives of
$[f]_k$. Then $[\tilde{f}_2-\tilde{f}_1]_{k+1}$ belongs to $\mathcal{I}_S^k/\mathcal{I}_S^{k+1}$ and
therefore $[\tilde{f}_1 h]_{k+1}-[\tilde{f}_2
h]_{k+1}=[0]_{k+1}$ since $h$ belongs to $\mathcal{I}_S/\mathcal{I}_S^{k+1}$.
\end{Remark}
\begin{Remark}
Theorem 2.1 of \cite{ABTEmbeddings} proves that $S$ is $k$-splitting if and only if there exists a $k$-splitting atlas,
i.e. an atlas $\{U_{\alpha},z_{\alpha}\}$ adapted to a complex submanifold $S$  such that:
\[\frac{\partial z^p_{\beta}}{\partial z^r_{\alpha}}\in\mathcal{I}_S^k\]
for all $r=1,\ldots,m,p=m+1,\ldots,n$ and for each pair of indices
$\alpha,\beta$ such that $U_{\alpha}\cap U_{\beta}\cap S\neq\emptyset$.
\end{Remark}
\begin{Definition}\label{DefFfaithful}
If $\F$ is foliation of $M$ of rank $l$ strictly smaller than the dimension of $S$, if we denote by
$\sigma^*$ the map
from
$\T_{M,S}$ to $\T_{S}$ given in Proposition \ref{Prop:Splitting}, we shall denote by $\F^\sigma$ the
coherent sheaf of
$\Ol_S$-modules given by
\[\F^{\sigma}:=\sigma^*(\F|_{S}).\]
We shall say that $\sigma^{*}$ is \textbf{$\F$-faithful outside an analytic subset} $\Sigma \subset
S$ if
$\F^{\sigma}$ is a regular foliation of $S$ of rank $l$ on $S \setminus\Sigma$. If $\Sigma =
\emptyset$ we shall
simply say that $\sigma^*$ is \textbf{$\F$-faithful}.
\end{Definition}

We refer to \cite{ABTHolMapFol} for a treatment of $\F$-faithfulness in the case of splittings.
Assume that $\sigma^*$ is $\F$-faithful: an interesting question is whether there exists an analogue of $\sigma^*$ from
$\T_{M,S(1)}$ to
$\T_{S(1)}$, which restricted to $\T_{M,S}$ coincides with $\sigma^*$; this would permit us to
project a transversal foliation to a foliation of the first infinitesimal neighborhood.
 
First of all, we can suppose we are working on a splitting submanifold $S$.
\begin{Definition}
 We will call the sheaf $\T_{M,S(k)}:=\T_M\otimes \Ol_{S(k)}$ the \textbf{restriction of the ambient
tangent sheaf to
the $k$-th infinitesimal neighborhood}. 
\end{Definition}
We look for a splitting of the following sequence:
\begin{equation}\label{2split}
0\rightarrow \T_{S(1)}\rightarrow \T_{M,S(1)}\rightarrow \mathcal{N}_{S(1)}\rightarrow 0,
\end{equation}
where $\N_{S(1)}$ is the quotient of the two modules.
\begin{Remark}
Let
$(U_{\alpha},z^1_{\alpha},\ldots,z^n_{\alpha})$ be a
coordinate system adapted to $S$. Please remember Remark \ref{StruttLog}; since $S$ is a
submanifold the ideal of $S$ is
generated by
$z^1_{\alpha},\ldots, z^r_{\alpha}$ and we have
that $\T_{S(1)}$ is generated in $U_{\alpha}$ by:
\[ [z^r_{\alpha}]_2\frac{\partial}{\partial z^s_{\alpha}},\frac{\partial}{\partial
z^{m+1}_{\alpha}},\ldots,\frac{\partial}{\partial z^n_{\alpha}},\]
for $r,s$ varying in $1,\ldots,m$ while $\T_{M,S(1)}$ is generated on $U_{\alpha}$ by
\[\frac{\partial}{\partial z^1_{\alpha}},\ldots,\frac{\partial}{\partial z^n_{\alpha}}.\]
Let $\partial_{r,\alpha}$ be the image of $\partial/\partial z^r_{\alpha}$ in $\N_{S(1)}$ and
$\omega^r_{\alpha}$ its
dual element.
Now let \[v=[f^k_{\alpha}]_2\frac{\partial}{\partial z^k_{\alpha}}\]
be a section of $\T_{M,S(1)}$; we can see that the image of $v$ in $\N_{S(1)}$ is nothing else than
$[f^r_{\alpha}]_1\partial_{r,\alpha}$.
We denote by $[v]$ the equivalence class of $v$ in $\N_{S(1)}$; please note that, given a function
$[g]_2$ in $\Ol_{S(1)}$ the $\Ol_{S(1)}$-module structure is given by
\[[g]_2\cdot[v]=[\theta_1([g]_2)\cdot v].\]
We compute now the transition functions of $\N_{S(1)}$; if we are in an atlas adapted to $S$ we have
that
$z^s_{\alpha}=h^s_{\alpha\beta,r}z^r_{\beta}$.
We have that:
\begin{align}
\partial_{r,\alpha}&=\bigg[\frac{\partial}{\partial z^r_{\alpha}}\bigg]=\bigg[\frac{\partial
z^k_{\beta}}{\partial
z^r_{\alpha}}\frac{\partial}{\partial z^k_{\beta}}\bigg]=\bigg[\frac{\partial z^s_{\beta}}{\partial
z^r_{\alpha}}\frac{\partial}{\partial z^s_{\beta}}\bigg]=\bigg[\frac{\partial
(h^{s}_{\alpha\beta,r'}z^{r'}_{\beta})}{\partial
z^r_{\alpha}}\frac{\partial}{\partial z^s_{\beta}}\bigg]\nonumber\\ 
&=\bigg[\frac{\partial
(h^{s}_{\alpha\beta,r'})}{\partial
z^r_{\alpha}}z^{r'}_{\beta}\frac{\partial}{\partial
z^s_{\beta}}\bigg]+\bigg[h^{s}_{\alpha\beta,r'}\delta_r^{r'}\frac{\partial}{\partial
z^s_{\beta}}\bigg]=[h^{s}_{\alpha\beta,r}]_2\partial_{s,\beta},\nonumber
\end{align}
where last equality comes from the equivalence relations that define $\N_{S(1)}$.
\end{Remark}
\begin{Remark}
Please note that the transition functions for $(\N_{S(1)})^*$, as an $\Ol_{S(1)}$-module are
given by:
\[\omega^s_{\beta}=[h^s_{\alpha\beta,r}]_2\omega^r_{\alpha}.\]
Please note that $(\N_{S(1)})^*$ is isomorphic to $\mathcal{I}_S/\mathcal{I}_S^2$ with the structure
of $\Ol_{S(1)}$
module given by the projection $\theta_1:\Ol_{S(1)}\to\Ol_S$. 
\end{Remark}

\begin{Lemma}\label{2Split}
Let $M$ be a $n$-dimensional complex manifold, $S$ a submanifold of codimension $r$. Then sequence \eqref{2split} splits
if
 $S$ is $2$-splitting, i.e. there exists an atlas adapted to $S$ such that:
\begin{equation}\label{Cociclo2TangBundle}
\bigg[\frac{\partial z^p_{\beta}}{\partial z^r_{\alpha}}\bigg]_2\equiv [0]_2,
\end{equation}
for $p=m+1,\ldots,n$ and $r=1,\ldots,m$.
\end{Lemma}
\begin{proof}
We have to compute the image in $H^1(S,\Hom(\mathcal{N}_{S(1)},\T_{S(1)}))$ through the
coboundary operator of cochain $\{U_{\alpha}\cap
S,\omega^r_{\alpha}\otimes\partial_{r,\alpha} \}$ representing the identity in $H^0(\mathcal{U},
\Hom(\mathcal{N}_{S(1)},\mathcal{N}_{S(1)})).$
We compute then:
\begin{align}
\delta(U_{\alpha},\omega^r_{\alpha}\otimes\partial_{r,\alpha})=&\omega^r_{\beta}\otimes\frac{\partial}{\partial
z^r_{\beta}}-\omega^s_{\alpha}\otimes\frac{\partial}{\partial z^s_{\alpha}}\nonumber\\
=&\omega^r_{\beta}\otimes\frac{\partial}{\partial z^r_{\beta}}-\bigg[\frac{\partial z^s_{\alpha}}{\partial
z^{r'}_{\beta}}\frac{\partial z^k_{\beta}}{\partial
z^s_{\alpha}}\bigg]_2\omega^{r'}_{\beta}\otimes\frac{\partial}{\partial
z^k_{\beta}}\nonumber\\
=&\omega^r_{\beta}\otimes\frac{\partial}{\partial z^r_{\beta}}-\bigg[\frac{\partial z^s_{\alpha}}{\partial
z^{r'}_{\beta}}\frac{\partial z^r_{\beta}}{\partial z^s_{\alpha}}\bigg]_2
\omega^{r'}_{\beta}\otimes\frac{\partial}{\partial z^r_{\beta}}\nonumber\\&-\bigg[\frac{\partial z^p_{\beta}}{\partial
z^s_{\alpha}}\frac{\partial z^s_{\alpha}}{\partial
z^{r'}_{\beta}}\bigg]_2\omega^{r'}_{\beta}\otimes\frac{\partial}{\partial
z^p_{\beta}}\nonumber\\
=&-\bigg[\frac{\partial z^p_{\beta}}{\partial
z^s_{\alpha}}\frac{\partial z^s_{\alpha}}{\partial
z^{r'}_{\beta}}\bigg]_2\omega^{r'}_{\beta}\otimes\frac{\partial}{\partial
z^p_{\beta}}.\label{ClassStrongSplit}
\end{align}
This class is clearly zero if we are using a $2$-splitting atlas.
\end{proof}
\begin{Remark}
In the last equality of the computation above there is marginal subtle point. If $S$ is $2$-splitting then it is
splitting. We saw above that this implies that in an atlas adapted to $S$ and to the splitting
\[\frac{\partial z^p_{\alpha}}{\partial z^r_{\beta}}\in \mathcal{I}_S, \quad\frac{\partial
z^r_{\alpha}}{\partial z^p_{\beta}}\in
\mathcal{I}_S.\]
We know also that:
\[\frac{\partial z^k_{\alpha}}{\partial z^r_{\beta}}\frac{\partial z^s_{\beta}}{\partial z^k_{\alpha}}=\delta_r^s.\]
Restricting ourselves to the $1$-st infinitesimal neighborhood we have that:
\[[\delta_r^s]_2=\bigg[\frac{\partial z^k_{\alpha}}{\partial z^r_{\beta}}\frac{\partial z^s_{\beta}}{\partial
z^k_{\alpha}}\bigg]_2=\bigg[\frac{\partial z^{r'}_{\alpha}}{\partial z^r_{\beta}}\frac{\partial z^s_{\beta}}{\partial
z^{r'}_{\alpha}}\bigg]_2+\bigg[\frac{\partial z^p_{\alpha}}{\partial z^r_{\beta}}\frac{\partial z^s_{\beta}}{\partial
z^p_{\alpha}}\bigg]_2=\bigg[\frac{\partial z^{r'}_{\alpha}}{\partial z^r_{\beta}}\frac{\partial z^s_{\beta}}{\partial
z^{r'}_{\alpha}}\bigg]_2,\]
using the splitting hypothesis.
\end{Remark}
\begin{Remark}
Looking at how we have constructed the splitting in the former Lemma, if $\tilde{\rho}$ is the
$\theta_1$-derivation associated to the splitting of $S$, we have that the splitting morphism
$\sigma^*:\T_{M,S(1)}\to\T_{S(1)}$ is given in an atlas adapted to the two splitting by:
\[f^k\frac{\partial}{\partial z^k}\mapsto \tilde{\rho}(f^r)\frac{\partial}{\partial
z^r}+f^p\frac{\partial}{\partial z^p}.\] 
\end{Remark}
Now the natural question is under which conditions the splitting of sequence \eqref{2split} is
equivalent to
the existence of a $2$-splitting atlas. It seems like the splitting of this sequence is not
enough.
Indeed, if we try to follow the usual approach in proving the argument, as in Proposition
\cite[Theorem 2.1]{ABTEmbeddings}, we have some problems.
The first thing we can remark is that the dual of $\T_{M,S(1)}$
is nothing else than $\Omega_M\otimes \Ol_{S(1)}$. Now, a splitting of \eqref{2split} implies there
exists a map
$\gamma$ from $\Omega_M\otimes \Ol_{S(1)}$ to $(\N_{S(1)})^*$  and since 
we remarked that $(\N_{S(1)})^*$ is isomorphic to $\mathcal{I}_S/\mathcal{I}_S^2$ as an
$\Ol_{S(1)}$-module, through
a map
\[\tau:(\N_{S(1)})^*\to \mathcal{I}_S/\mathcal{I}_S^2.\]
this gives rise to a splitting of the map
\[d_2:\mathcal{I}_S/\mathcal{I}_S^2\to\Omega_M\otimes \Ol_{S(1)}\]
which sends an element $[f]_2$ of $\mathcal{I}_S/\mathcal{I}_S^2$ into $df\otimes [1]_2$.
Now, there exists a well defined map $d_3$ from $\Ol_M/\mathcal{I}_S^3$ to $\Omega_M\otimes
\Ol_{S(1)}$, which
sends a class
$[f]_3$ to $d\tilde{f}\otimes [1]_2$.
The big problem is that, even if we suppose $S$ comfortably embedded (\cite{ABTEmbeddings}), i.e.,
the sequence
\[\xymatrix{0\ar[r]&\mathcal{I}_S^2/\mathcal{I}_S^3
\ar[r]&\mathcal{I}_S/\mathcal{I}_S^3\ar[r]&\mathcal{I}_S/\mathcal{I}_S^2 \ar[r] & 0}\]
splits, we have that the splitting of
\eqref{2split} only gives us a map between $\Ol_{S(1)}$ and the image through the splitting
$\nu:\mathcal{I}_S/\mathcal{I}_S^3\to\mathcal{I}_S/\mathcal{I}_S^2$ of
$\mathcal{I}_S/\mathcal{I}_S^2$ in $\mathcal{I}_S/\mathcal{I}_S^3$ and this map is not surjective.
Therefore it is
not a $\theta_{2,1}$-derivation splitting the short exact sequence of morphisms of rings
\[\xymatrix{0\ar[r]&\mathcal{I}_S/\mathcal{I}_S^3\ar[r]&\Ol_{S(1)}\ar[r]&\Ol_S\ar[r]&0.}\]
\begin{Remark}
To solve this problem we could find under which conditions there exists a splitting of the map
$\tilde{d}_3:\mathcal{I}_S^2/\mathcal{I}_S^3\to\Omega_{M,S(1)}$. Using such a splitting and the
comfortable embedding we
could
find
a $\theta_{2,1}$-derivation splitting $\iota:\mathcal{I}_S/\mathcal{I}_S^3\to\Ol_{S(1)}$.
\end{Remark}

We define now a notion parallel to the one in Definition \ref{DefFfaithful}.
\begin{Definition}
If $\F$ is foliation of $M$ of rank $l$ strictly smaller than dimension $S$, if we denote by $\sigma^*_2$ the map from
$\T_{M,S(1)}$ to $\T_{S(1)}$ splitting sequence \eqref{2split}, we shall denote by $\F^{\sigma_2}$ the coherent sheaf of
$\Ol_S(1)$-modules given by
\[\F^{\sigma_2}:=\sigma^*_2(\F|_{S(1)}).\]
We shall say that $\sigma^{*}_2$ is \textbf{first order $\F$-faithful outside an analytic subset}
$\Sigma \subset S$ if
$\F^{\sigma_2}$
is a regular foliation of $S(1)$ of rank $l$ on $S \setminus\Sigma$. If $\Sigma = \emptyset$ we shall simply say that
$\sigma^*_2$ is \textbf{first order $\F$-faithful}.
\end{Definition}
We state some results coming from \cite{ABTHolMapFol}, in particular Lemma 7.5 and Lemma 7.6 about $\F$-faithfulness.
\begin{Lemma}[\cite{ABTHolMapFol}, Lemma 7.5]\label{Lemma:Faithfulness}
Let $S$ be a splitting submanifold of a complex manifold $M$ , and let $\F$ be a
holomorphic foliation on $M$ of dimension equal to $1$ or to the dimension of $S$. If there
exists $x_0\in  S \setminus \Sing(F)$ such that $\F$ is tangent to $S$ at $x_0$ , i.e., $(\F|_S)x_0\subset \T_{S,x_0}$,
then any splitting morphism is $\F$-faithful outside a suitable analytic subset of $S$.
\end{Lemma}
\begin{Lemma}[\cite{ABTHolMapFol}, Lemma 7.6]
Let $S$ be a non-singular hypersurface splitting in a complex manifold $M$ , and
let $\F$ be a one dimensional holomorphic foliation on $M$. Assume that $S$ is not contained
in $\Sing(F)$. Then there is at most one splitting morphism $\sigma^*$ which is not $\F$-faithful outside
a suitable analytic subset of $S$.
\end{Lemma}
Indeed, speaking of first order $\F$-faithfulness we have a simple results which gives us some
insight.
\begin{Lemma}
Let $S$ be a submanifold $2$-splitting in a complex manifold $M$ , and
let $\F$ be a one dimensional holomorphic foliation on $M$. Let $\sigma_2^*:\T_{M,S(1)}\to\T_{S(1)}$ be the splitting
morphism. If $\sigma_2^*|_S$ is $\F$-faithful outside an analytic subset $\Sigma$, then $\sigma_2^*$ is
first order $\F$-faithful outside $\Sigma$.
\end{Lemma}
\begin{proof}
We check that $\F^{\sigma_2}$ satisfies the requests of Definition \ref{Def:FolInf}. By hypothesis $\F^{\sigma_2}|_S$ is
a foliation of $S$. Since the rank of $\F^{\sigma_2}$ is $1$ it is an involutive subbundle of $\T_{S(1)}$; moreover,
for each point $x\in S\setminus \Sigma$ we can find a generator $v$ of $\F^{\sigma_2}_x$ such that $v|_S$ is non zero.
Therefore, $\T_{S(1),x}/\F^{\sigma_2}_x$ is $\Ol_{S(1)}$-free. 
\end{proof}
Directly from this last Lemma and Lemma \ref{Lemma:Faithfulness} we have that.
\begin{Corollary}
Let $S$ be a splitting submanifold of a complex manifold $M$ , and let $\F$ be a
holomorphic foliation on $M$ of dimension equal to $1$ or to the dimension of $S$. If there
exists $x_0\in  S \setminus \Sing(F)$ such that $\F$ is tangent to $S$ at $x_0$ , i.e., $(\F|_S)x_0\subset \T_{S,x_0}$,
then any $2$-splitting morphism is first order $\F$-faithful outside a suitable analytic subset of
$S$.
\end{Corollary}

\section{Extension of foliations and embedding in the normal bundle}\label{SecExtension}
\begin{Definition}
Let $S$ be a codimension $m$ submanifold of a $n$-dimensional complex manifold $M$. Let $S(1)$ be its first
infinitesimal neighborhood and $S_{N}(1)$ be the first infinitesimal neighborhood of its embedding as the zero section
of its normal bundle in $M$. We denote by $\mathcal{O}_{N_S}$ the structure sheaf of the normal bundle of $S$ and by $
\mathcal{I}_{S,N_S}$ the ideal sheaf of $S$ in $N_S$.
We say $S_{N}(1)$ is \textbf{isomorphic} to $S(1)$ if there exists an isomorphism
$\phi:\mathcal{O}_{N_S}/\mathcal{I}^2_{S,N_S}\to\mathcal{O}_M/\mathcal{I}_S^2$ such that $\theta_1\circ\phi=\theta_1^N$,
where
$\theta_1:\mathcal{O}_M/\mathcal{I}_S^2\to\mathcal{O}_S$ and
$\theta_1^N:\mathcal{O}_{N_S}/\mathcal{I}^2_{S,N_S}\to\mathcal{O}_S$
are the canonical projections.
\end{Definition}
\begin{Proposition}[\cite{ABTHolMapFol}~Prop. 1.3]
Let $S$ be a submanifold of a complex manifold $M$. Then $S$ splits into $M$ if and only if its first infinitesimal
neighborhood $S(1)$ in $M$ is isomorphic to its first infinitesimal neighborhood $S_N(1)$ in $N_S$, where we are
identifying $S$ with the zero section of $N_S$.
\end{Proposition}
\begin{Remark}\label{InfNeigh}
In general, given a vector bundle $E$ over a submanifold $S$, we have that $TE|_S$ is canonically isomorphic to
$TS\oplus E$. When $E$ is $N_S$ this implies that the projection on the second summand of
$TN_S|_S=TS\oplus N_S$ gives
rise to an isomorphism of $N_S$ and $N_{0_S}$, i.e., the normal bundle of $S$ as the zero section of $N_S$. Therefore we
have
an isomorphism between $\mathcal{I}_S/\mathcal{I}_S^2$ and $\mathcal{I}_{S,N_S}/\mathcal{I}_{S,N_S}^2.$
\end{Remark}

Let $S$ be a codimension $m$ submanifold of an $n$-dimensional complex manifold $M$ and let $\F$ be a foliation of
$S$.
Thanks to the Holomorphic Frobenius theorem, we know that there exists an atlas $\{(U_{\alpha};
z^1_{\alpha},\ldots,z^n_{\alpha})\}$
adapted to $S$ and $\F$. In such an atlas we know that $\F=\ker(dz_{\alpha}^{m+l+1}|_S,\ldots, dz_{\alpha}^n|_S)$.
An equivalent formulation of the Frobenius theorem states that a submodule of $\Omega^1(S)$ is integrable if and only if
each stalk is generated by exact forms.
We denote by $\pi:N_S\to S$ the normal bundle of $S$.
The map $\pi$ is holomorphic, therefore $\pi^*(dz_{\alpha}^{k}|_S)$ is a well defined local holomorphic $1$-form on
$\pi^{-1}(U_{\alpha})\subset N_S$. Moreover, since
$\{dz_{\alpha}^{m+l+1}|_S,\ldots, dz_{\alpha}^n|_S\}$ is an integrable system of $1$-forms, so is
$\{\pi^*(dz_{\alpha}^{m+l+1}|_S),\ldots, \pi^*(dz_{\alpha}^n|_S)\}$.
Then $\{\pi^*(dz_{\alpha}^{m+l+1}|_S),\ldots, \pi^*(dz_{\alpha}^n|_S)\}$ defines a foliation $\tilde{\F}$ of $N_S$,
whose leaves are the preimages of the leaves of $\mathcal{F}$ through $\pi$. Since $S$ is regular $TM$ is trivialized
on each coordinate neighborhoods and so is $N_S$.
In the following we use the atlas $\{(\pi^{-1}(U_{\alpha}),v_{\alpha}^1,\ldots,
v^m_{\alpha},z^{m+1}_{\alpha},\ldots,z^{n}_{\alpha}\}$ of $N_S$ given by the trivializations of the normal bundle, where
$v^r_{\alpha}$ are the coordinates in the fiber; then $\tilde{\F}$ is generated on $\pi^{-1}(U_{\alpha})$ by 
\[\frac{\partial}{\partial v^1_{\alpha}},\ldots,\frac{\partial}{\partial v^m_{\alpha}},\frac{\partial}{\partial
z^{m+1}_{\alpha}}\bigg|_S,\ldots,\frac{\partial}{\partial z^{m+l}_{\alpha}}\bigg|_S.\]
The fibers of $\pi$ are the leaves of a holomorphic foliation of $N_S$, called the \textbf{vertical foliation}, which we
denote
by $\mathcal{V}$.
On $\pi^{-1}(U_{\alpha})$ it is generated by 
\[\frac{\partial}{\partial v^1_{\alpha}},\ldots,\frac{\partial}{\partial v^m_{\alpha}}.\]
We study now the splitting of the following sequence, when restricted to the first infinitesimal neighborhood of the
embedding of $S$ as the zero section of $N_S$:
\begin{equation}\label{folest}
\xymatrix{0 \ar[r] &\mathcal{V} \ar[r]^{\iota} &\tilde{\F} \ar[r]^{\pr} &\tilde{\F}/\mathcal{V} \ar[r] & 0}.
\end{equation}
A result of Grothendieck \cite{Grothendieck} tells us that the splitting of the sequence is equivalent
to the vanishing of a cohomology class in $H^1(M,\Hom(\tilde{\F}/\mathcal{V},\mathcal{V}))$.
The splitting of this sequence is equivalent to the fact that there exist an isomorphism
$\tilde{\F}\simeq\mathcal{V}\oplus\tilde{\F}/\mathcal{V}$ compatible with the projection $\pr$ and the map $\iota$.
Even if this result was already used implicitly in Section \ref{SecAtiyah} we sketch a proof to show how it can be
used
operatively in our work.
Indeed, let $\omega$ be the cohomology class associated to the splitting of a short exact sequence of sheaves 
\begin{equation}\label{SecEsCort}
0\to \E\to\F\to\G\to 0;
\end{equation}
this obstruction is the image of the identity homomorphism in $H^0(M,\Hom(\G,\G))$ into $H^1(M,\Hom(\G,\E))$.
By the Long Exact Sequence Theorem for \u{C}ech cohomology we compute $\omega$ in the following way: let
$\{U_{\alpha},\textrm{Id}\}$ be the class representing the identity in 
$H^0(M,\Hom(\G,\G))$, we take a lift $(U_{\alpha},\tau_{\alpha})$ in $C^0(\mathcal{U},\Hom(\G,\F))$ and take its
\u{C}ech coboundary, $\{U_{\alpha\beta}, \tau_{\beta}-\tau_{\alpha}\}$. Clearly,
$\pr\circ\tau_{\beta}-\pr\circ\tau_{\alpha}=0$, so this is a well defined element of $C^1(\mathcal{U},\Hom(\G,\E))$. By
diagram chasing, it is shown this is a \u{C}ech cocycle which represents $\omega$.
Suppose now $\omega$ is $0$ in cohomology: this means there exists a cochain
$\{U_{\alpha},\sigma_{\alpha}\}$ in
$C^0(\mathcal{U},\Hom(\G,\E))$ whose coboundary is $\omega$, i.e.
$\sigma_{\beta}-\sigma_{\alpha}=\tau_{\beta}-\tau_{\alpha}$.
We define now a \u{C}ech cochain in $C^0(\mathcal{U},\Hom(\E\oplus\G,\F))$ as $\{U_{\alpha},\theta_{\alpha}\}$ where
$\theta_{\alpha}$ is defined on each $U_{\alpha}$ as:
\[\theta_{\alpha}:(v,w)\mapsto (\iota(v-\sigma_{\alpha}(w))+\tau_{\alpha}(w)).\]
We compute now $\delta\{U_{\alpha},\theta_{\alpha}\}$; on each $U_{\alpha\beta}$:
\begin{align*}
&\iota(v-\sigma_{\beta}(w))+\tau_{\beta}(w)-\iota(v-\sigma_{\alpha}(w))+\tau_{\alpha}(w)\\&=\iota(\sigma_{
\alpha }(w)-\sigma_{\beta}(w))+\tau_{\beta}(w)-\tau_{\alpha}(w)=0.
\end{align*}
So, we have a global isomorphism of sheaves between $\E\oplus\G$ and $\F$ satisfying our requests.
\begin{Remark}
Please note that $\tilde{\F}/\mathcal{V}$, when restricted to $S$ is nothing else but the foliation $\F$.
This follows directly from our construction of $\tilde{\F}$ as the pull-back foliation defined by the
integrable system $\{\pi^*(dz_{\alpha}^{m+l+1}|_S),\ldots, \pi^*(dz_{\alpha}^n|_S)\}$.
\end{Remark}
\begin{Lemma}\label{LemmaFolEstNormImplFolEst}
Let $S$ be a splitting submanifold in $M$. If there exists a foliation of the first infinitesimal neighborhood of the
embedding of $S$ as the zero section of its normal bundle, then there exists a foliation of the first infinitesimal
neighborhood of $S$ embedded in $M$.
\end{Lemma}
\begin{proof}
If there exists a foliation of the first infinitesimal neighborhood of the embedding of $S$ as the zero section of its
normal bundle we can find an atlas of $N_S$ given by $\{V_{\alpha},
u^{1}_{\alpha},\ldots,u^{m}_{\alpha},z^{m+1}_{\alpha},\ldots,z^n_{\alpha}\}$  such that, if
$V_{\alpha}\cap
V_{\beta}\cap S\neq\emptyset$ we have that
\[\bigg[\frac{\partial u^r_{\alpha}}{\partial z^i_{\beta}}\bigg]_2=[0]_2;\quad \frac{\partial z^{t'}_{\alpha}}{\partial
z^i_{\beta}}=0\]
where $r=1,\ldots,m$ and $t'=m+l+1,\ldots,n$.\\
We use the isomorphism $\phi:\mathcal{O}_{N_S}/\mathcal{I}^2_{S,N_S}\to\mathcal{O}_M/\mathcal{I}_S^2$, taking the images
\[
[\tilde{z}^1_{\alpha}]_2=\phi([u^1_{\alpha}]_2),\ldots,[\tilde{z}^r_{\alpha}]_2=\phi(u^r_{\alpha}),[\tilde{z}^{m+1}_{
\alpha}]_2=\phi([z^{m+1}_{\alpha}]_2),\ldots,[\tilde{z}^{n}_{\alpha}]_2=\phi(z^n_{\alpha});\] 
there exists open sets $U_{\alpha}\supset\pi(V_{\alpha})$ (modulo shrinking) where we can choose
representatives of these classes
such that
$(U_{\alpha},\tilde{z}^1_{\alpha},\ldots,\tilde{z}^n_{\alpha}),$
is a coordinate system adapted to $S$ and $\F$.
If $U_{\alpha}\cap U_{\beta}\cap S\neq\emptyset$ we can check that, since $\partial/\partial
\tilde{z}_{\beta}^{m+1},\ldots,\partial/\partial \tilde{z}^{m+l}_{\alpha}$ are logarithmic
\[\frac{\partial[\tilde{z}^r_{\beta}]_2}{
\partial\tilde{z}^i_{\alpha}}=\bigg[\frac{\partial\tilde{z}^r_{\beta}}{\partial\tilde{z}^i_{\alpha}}\bigg]_2=\bigg[\frac
{\partial u^r_{\beta}}{\partial z^i_{\alpha}}\bigg]_2=[0] _2,\]
for $r=1,\ldots,m$ and $i=m+1,\ldots,m+l+1$.
Following the same line of thought
\[\frac{\partial[\tilde{z}^{t'}_{\beta}]_2}{
\partial\tilde{z}^i_{\alpha}}=\bigg[\frac{\partial\tilde{z}^{t'}_{\beta}}{\partial\tilde{z}^i_{
\alpha } } \bigg]_2=\bigg[\frac
{\partial z^{t'}_{\beta}}{\partial z^i_{\alpha}}\bigg]_2=[0] _2,\]
for $t'=m+l+1,\ldots,n$ and $i=m+1,\ldots,m+l+1$.
\end{proof}
So, the problem of extending a foliation outside a submanifold boils down in the splitting case to
understand when \eqref{folest} splits and the image through the splitting of
$\F/\mathcal{V}$ is involutive.
We start by finding a sufficient condition for this to happen.
\begin{Proposition}
Let $M$ be a complex manifold of dimension $n$, and $S$ a splitting codimension $m$ submanifold. Let $\F$ be
a foliation of $S$ and $\pi: N_S\to M$ the normal bundle of $S$ in $M$. Let $\tilde{\F}=\pi^*(\F)$ and
$\mathcal{V}$ the vertical foliation given by $\ker d\pi$.
The sequence:
\begin{equation*}
\xymatrix{0 \ar[r] &\mathcal{V} \ar[r]^{\iota} &\tilde{\F} \ar[r]^{\pr} &\tilde{\F}/\mathcal{V} \ar[r] & 0}
\end{equation*}
splits if there exists an atlas adapted to $\F$ and $S$ such that
\[\frac{\partial^2 z^r_{\alpha}}{\partial z^i_{\beta}\partial z^s_{\beta}}\in \mathcal{I}_S,\]
for all $r,s=1,\ldots,m$ and $i=m+1,\ldots,m+l$.
\end{Proposition}
\begin{proof}
We compute the obstruction to the splitting of the sequence, following \cite{Atiyah} and \cite{Grothendieck}: we apply
the functor
$\Hom(\tilde{\F}/\V,\cdot)$ to sequence \eqref{folest} and compute the image of the identity through the coboundary map
\[\delta: H^0(S,\Hom(\tilde{\F}/\V,\tilde{\F}/\mathcal{V})) \to H^1(S,\Hom(\tilde{\F}/\V,\mathcal{V})).\]
We fix an atlas $\{U_{\alpha},z_{\alpha}\}$ adapted to $S$ and $\F$ and we denote the quotient frame for
$\tilde{\F}/\V$ by $\{\partial_{m+1,\alpha},\ldots,\partial_{m+l,\alpha}\}$  (i.e.,
$\partial_{m+1,\alpha}$ is the equivalence class of $\partial/\partial z^{m+1}_{\alpha}|_S$) and by
$\{\omega^{m+1}_{\alpha},\ldots,\omega^{m+l}_{\alpha}\}$ its dual frame. The cocycle representing the identity in
$H^0(S,\Hom(\tilde{\F}/\V,\tilde{\F}/\mathcal{V}))$ is then represented as
$\{U_{\alpha},\omega^j_{\alpha}\otimes\partial_{j,\alpha}\}$; the obstruction to the splitting of the sequence is then
\begin{align}
\delta\{\omega^j_{\alpha}\otimes\partial_{j,\alpha}\}&=\omega^j_{\beta}\otimes\frac{\partial}{\partial
z^j_{\alpha}}-\omega^j_{\alpha}\otimes\frac{\partial}{\partial z^j_{\alpha}}\nonumber\\
&=\omega^j_{\beta}\otimes\frac{\partial}{\partial z^j_{\alpha}}-\bigg[\frac{\partial z^j_{\alpha}}{\partial
z^i_{\beta}}\frac{\partial z^{j'}_{\beta}}{\partial z^j_{\alpha}}\bigg]_2
\omega^i_{\beta}\otimes\frac{\partial}{\partial z^{j'}_{\beta}}-\bigg[\frac{\partial z^j_{\alpha}}{\partial
z^i_{\beta}}\frac{\partial v^{r}_{\beta}}{\partial z^j_{\alpha}}\bigg]_2\omega^i_{\beta}\otimes\frac{\partial}{\partial
v^{r}_{\beta}}\nonumber\\
&=-\bigg[\frac{\partial z^j_{\alpha}}{\partial z^i_{\beta}}\frac{\partial v^{r}_{\beta}}{\partial
z^j_{\alpha}}\bigg]_2\omega^i_{\beta}\otimes\frac{\partial}{\partial v^{r}_{\beta}}.\label{classFol1}
\end{align}
The vanishing of \eqref{classFol1} is a sufficient condition for the splitting of the sequence; this class vanishes if
$\partial v^{r}_{\beta}/\partial z^j_{\alpha}$ belong to $\mathcal{I}^2_{N_S}.$
Moreover, the coordinate changes map of $N_S$ have a peculiar structure:
\[v^r_{\beta}=v^s_{\alpha}\frac{\partial z^r_{\alpha}}{\partial z^s_{\beta}}.\]
Therefore:
\begin{align}
-\bigg[\frac{\partial z^j_{\alpha}}{\partial z^i_{\beta}}\frac{\partial v^{r}_{\beta}}{\partial
z^j_{\alpha}}\bigg]_2&=-\bigg[\frac{\partial z^j_{\alpha}}{\partial z^i_{\beta}}\frac{\partial}{\partial
z^j_{\alpha}}\bigg(\frac{\partial z^r_{\alpha}}{\partial z^s_{\beta}}\bigg)\bigg]_2\nonumber\\
&=-\bigg[v^s_{\alpha}\frac{\partial
z^j_{\alpha}}{\partial z^i_{\beta}}\frac{\partial^2 z^r_{\alpha}}{\partial z^j_{\alpha}\partial
z^s_{\beta}}\bigg]_2=-\bigg[v^s_{\alpha}\frac{\partial
z^j_{\alpha}}{\partial z^i_{\beta}}\frac{\partial^2 z^r_{\alpha}}{\partial z^s_{\alpha}\partial
z^j_{\beta}}\bigg]_2\nonumber;
\end{align}
using the isomorphism between $\mathcal{I}_S/\mathcal{I}_S^2$ and $\mathcal{I}_{S,N_S}/\mathcal{I}_{S,N_S}^2$ we see
that the last expression vanishes if $\partial z^s_{\alpha}/\partial z^j_{\beta}\in\mathcal{I}_S^2.$ 
\end{proof}
\begin{Remark}
Since we are working in an atlas of $N_S$ adapted to $S$ and $\F$ we have that
\[\frac{\partial z^p_{\alpha}}{\partial v^r_{\beta}}\equiv 0\quad,\quad \frac{\partial
z^t_\alpha}{\partial
z^i_{\beta}}\equiv 0,\]
for $p=m+1,\ldots,n$, $r=1,\ldots,m$, $t=m+l+1,\ldots,n$, $i=m+1,\ldots,n$ (please remark that we
are \textit{not} following our usual convention).
Looking at the transition functions of the tangent bundle of $N_S$, in an atlas adapted to $S$ and
$\F$, on the first infinitesimal neighborhood of the embedding of $S$ as the zero section of $N_S$,
the following equality holds:
\[ [\delta^i_j]_2=\bigg[\frac{\partial v^r_{\alpha}}{\partial
z^j_{\beta}}\frac{\partial
z^i_{\beta}}{\partial
v^r_{\alpha}}+\frac{\partial z^p_{\alpha}}{\partial z^j_{\beta}}\frac{\partial z^i_{\beta}}{\partial
z^p_{\alpha}}\bigg]_2,\]
now, since ${\partial z^p_{\alpha}}/{\partial v^r_{\beta}}\equiv 0$ we have that
\[[\delta^i_j]_2=\bigg[\frac{\partial z^p_{\alpha}}{\partial z^j_{\beta}}\frac{\partial
z^i_{\beta}}{\partial
z^p_{\alpha}}\bigg]_2\]
and since ${\partial z^t_\alpha}/{\partial z^i_{\beta}}\equiv 0$ for $t=m+l+1,\ldots,n$,
$i=m+1,\ldots,n$ we have that
\[[\delta^i_j]_2=\bigg[\frac{\partial
z^{i'}_{\alpha}}{\partial z^j_{\beta}}\frac{\partial z^i_{\beta}}{\partial
z^{i'}_{\alpha}}\bigg]_2,\]
where $i,j=m+1,\ldots,m+l$ and we sum over $i'=m+1,\ldots,m+l$. 
\end{Remark}
\begin{Lemma}\label{LemmaExtNonInv}
 Let $S$ be a splitting submanifold of $M$, let $\F$ be foliation of $S$; if the sequence 
\begin{equation*}
\xymatrix{0 \ar[r] &\mathcal{V} \ar[r]^{\iota} &\tilde{\F} \ar[r]^{\pr} &\tilde{\F}/\mathcal{V} \ar[r] & 0}.
\end{equation*}
splits on the first infinitesimal neighborhood of $S$ embedded as the zero section of its normal bundle in $M$, then
$\F$ extends as a subsheaf of $\T_{S(1)}$.
\end{Lemma}
\begin{proof}
Suppose we are working in an atlas adapted to $S$ and $\F$; on $S$ we have the following isomorphism:
\begin{equation}\label{EquationFoliSplittingOnS}
\tilde{\F}|_S=\mathcal{V}|_S\oplus \F.
\end{equation}
This follows directly from our construction of $\tilde{\F}$ as the pull-back foliation defined by the
integrable system $\{\pi^*(dz_{\alpha}^{m+l+1}|_S),\ldots, \pi^*(dz_{\alpha}^n|_S)\}$.
Therefore we have that $\tilde{\F}/\mathcal{V}|_S\simeq\F$ and this implies that the cochain
representing \eqref{classFol1} vanishes identically when restricted to $S$. Therefore we know that the components
of the cochain $\{U_{\alpha},\sigma_{\alpha}\}$ are identically $0$ when restricted to
$S$.
Let $v$ be a section of $\tilde{\F}/\mathcal{V}$; its image $\tau_{\alpha}-\sigma_{\alpha}(v)$ is a section of
$\tilde{\F}$. From the discussion above we can remark is that $\sigma_{\alpha}(v)|_S\equiv 0$;
moreover, since on $S$
\eqref{EquationFoliSplittingOnS} holds we have that $\tau_{\alpha}(v)|_S\in\F\subset \T_S$ and this proves that $v$
belongs to $\T_{S(1)}$.
\end{proof}
\begin{Corollary}
Let $S$ be a splitting submanifold of $M$, let $\F$ be a rank $1$ foliation of $S$; if the sequence 
\begin{equation*}
\xymatrix{0 \ar[r] &\mathcal{V} \ar[r]^{\iota} &\tilde{\F} \ar[r]^{\pr} &\tilde{\F}/\mathcal{V} \ar[r] & 0}.
\end{equation*}
splits on the first infinitesimal neighborhood of $S$ embedded as the zero section of its normal bundle in $M$, then
$\F$ extends as a foliation of the first infinitesimal neighborhood of $S$ in $M$.
Moreover, we can find an atlas adapted to $S$ and $\F$ given by a collection of charts
$\{U_{\alpha},(v^1_{\alpha},\ldots,v^m_{\alpha},z^{m+1}_{\alpha}\ldots,z^n_{\alpha})\}$ such that the class
\eqref{classFol1} can be represented by the $0$ cochain.
\end{Corollary}
\begin{proof}
If $\tilde{\F}/\V$ has rank $1$ we have that its image through the splitting morphism of \eqref{folest} is a rank $1$
(therefore involutive) subbundle of $\T_{S_N(1)}$.
Thanks to Lemma \ref{LemmaFolEstNormImplFolEst} we have the first part of the assertion.
Corollary \ref{Frobenius} gives us the second part of the assertion.
\end{proof}
\begin{Corollary}\label{FirstOrder}
Let $M$ be a $n$-dimensional complex manifold, $S$ a codimension $m$ splitting submanifold, $\F$ a
regular foliation of
$S$. Suppose $S$ admits first order extendable tangent bundle, then $\F$ extends to a subsheaf of
$\T_{S(1)}$.
\end{Corollary}
\begin{proof}
We work in an atlas adapted to $S$ and $\F$; by Corollary \ref{Frobenius} first order extendable
tangent bundle implies 
\[\bigg[\frac{\partial z^r_{\alpha}}{\partial z^p_{\beta}}\bigg]_2=0\] 
which in turn implies
\[\frac{\partial^2 z^r_{\alpha}}{\partial z^s_{\beta} \partial z^p_{\beta}}\in \mathcal{I}_S\]
and therefore that
\[\frac{\partial^2 z^r_{\alpha}}{\partial z^s_{\beta} \partial z^i_{\beta}}\in \mathcal{I}_S,\]
for $r,s=1,\ldots,m$, $i=m+1,\ldots,m+l$, which in turn implies the vanishing of \eqref{classFol1}
and the splitting of sequence \eqref{folest}; the extension of $\F$ is then given by the image of
$\tilde{\F}/\mathcal{V}$ in $\tilde{\F}$.
\end{proof}
\begin{Remark}
The reason why the splitting of \eqref{folest} is not a sufficient condition for the foliation to extend to the first
infinitesimal neighborhood lies in the fact that the image of $\tilde{\F}/\mathcal{V}$ may not be
involutive.
If this image is involutive we have a statement similar to the one in the last Corollary; anyway even if it is not
involutive, thanks to the results in section \ref{SectionNonInvolutive}, the splitting of \eqref{folest} is enough to
get some important insights on the Khanedani-Lehmann-Suwa action.
\end{Remark}
\begin{Remark}\label{Rem:ExtGenerators}
We want to see what happens in coordinates when we can extend the foliation.
First of all, the vanishing of the class \eqref{classFol1} in cohomology means there exists a cochain
$\{U_{\alpha},\sigma_{\alpha}\}\in C^{0}(S_N(1),(\F/\mathcal{V}))^*\otimes \mathcal{V})$ such that:
\begin{equation*}
\sigma_{\beta}-\sigma_{\alpha}=-\bigg[\frac{\partial z^j_{\alpha}}{\partial z^i_{\beta}}\frac{\partial
v^{r}_{\beta}}{\partial z^j_{\alpha}}\bigg]_2\omega^i_{\beta}\otimes\frac{\partial}{\partial v^{r}_{\beta}}.
\end{equation*}
In a coordinate system adapted to $S$ and $\F$ on each $U_{\alpha}$ we can write the elements of the cochain as
\[\sigma_{\alpha}=[c^s_{j,\alpha}]_2\omega^j_{\alpha}\otimes\frac{\partial}{\partial v^s_{\alpha}}.\]
Since the sequence splits when it is restricted to $S$ we can assume that the coefficients $c^s_{j,\alpha}$ of
each $\sigma_{\alpha}$ belong to $\mathcal{I}_S/\mathcal{I}_S^2$.
Without loss of generality we can suppose the local lifts $\tau_{\alpha}$ send the generators of
$\tilde{\F}/\mathcal{V}$, that we denote by $\partial_{i,\alpha}$, in the coordinate fields $\partial/\partial
z^i_{\alpha}$ (the difference about two different choices of lifts is absorbed by the cochain).
Then a generator $\partial/\partial z^i_{\alpha}|_S$ of $\F$ on $U_{\alpha}$ extends to the section $v$ of $\T_{S(1)}$
given by:
\[-[c^s_{j,\alpha}]_2\frac{\partial}{\partial v^s_{\alpha}}+\frac{\partial}{\partial z^j_{\alpha}}.\]
\end{Remark}

\section{Action of subsheaves of $\F$ on $\N_{\F,M}$}\label{SectionNonInvolutive}
As usual let $\F$ be a foliation of $S$: in this section we shall discuss how the existence of
coherent subsheaves of
$\T_{S(1)}$ that restricted to $S$ are subsheaves of $\F$ gives rise to variation actions on
$\N_{\F,M}$.
\begin{Lemma}
Let $\E$ be a coherent subsheaf of $\T_{S(1)}$ that, restricted to $S$, is a subsheaf of $\F$. Then
$\E$ is a subsheaf
of $\T_{M,S(1)}^{\F}$. 
\end{Lemma}
\begin{proof}
Let $\{U_{\alpha},z_{\alpha}\}$ be an atlas adapted to $S$ and $\F$. On each coordinate chart, a
section $v$ of $\E$ can
be written as:
\[[a^u]_2\frac{\partial}{\partial z^u}+[a^i]_2\frac{\partial}{\partial z^i},\]
with $a^u\in\mathcal{I}_S$. Therefore, thanks to Remark \ref{incoord}, we know that $v$ belongs to
$\T_{M,S(1)}^{\F}$.
\end{proof}
\begin{Definition}
Let $\E$ be a coherent subsheaf of $\T_{S(1)}$; we say it is
\textbf{$S$-faithful}\index{$S$-faithful} if the restriction map $|_S:\E\to \E|_S$ is
injective.
\end{Definition}

\begin{Proposition}\label{Prop:NonInvolutive}
Suppose $\E$ is a coherent subsheaf of $\T_{S(1)}$ that, restricted to $S$, is a subsheaf of $\F$,
that is
generated on an open set $U_{\alpha}$ by $\tilde{v}_{1,\alpha},\ldots,\tilde{v}_{k,\alpha}$
and is $S$-faithful.
Then there exists a partial holomorphic connection $(\delta,\E)$ for $\N_{\F,M}$.
\end{Proposition}
\begin{proof}
Since there are no generators sent to $0$ by the restriction to $S$,
then $\E|_{S\cap U_{\alpha}}$ is generated by $v_{k,\alpha}:= \tilde{v}_{k,\alpha}|_S$.
Please keep in mind that the generators of $\mathcal{E}|_S$ are always the restriction of the
generators of
$\mathcal{E}$, so, chosen the local generators of $\E$ we have a canonical way to extend the local
generators of
$\E|_S$.

Let $\pi$ be the projection from $\T_{M,S(1)}^{\F}$ to $\mathcal{A}$ and $w$ a section of $\E|_S$;
we define a map
$\tilde{\pi}:\E|_S\to\mathcal{A}$ by $\pi(w):=\pi(\tilde{w}),$ where $\tilde{w}$ is an extension of
$w$ as a
section of $\E$.
On a trivializing neighborhood for $\E$ a section has the following form: $w=[f^k]_1 v_{k,\alpha}\in
\E|_{S\cap
U_{\alpha}}$.
The difference between two representatives $\tilde{w}_1$ and $\tilde{w}_2$ of $w$ in $\E$ on
$U_{\alpha}$ can be written
in the following form:
\[[g^k]_2 \tilde{v}_{k,\alpha}\]
where the $g^k$ belong to $\mathcal{I}_S/\mathcal{I}_S^2$ and therefore belongs to
$\mathcal{I}_S\T_{M,S(1)}^{\F}$. 
Therefore the map $\tilde{\pi}$ does not depend on the extension chosen.

Suppose now we have a section $w$ of $\mathcal{E}|_S$ and two coordinate charts $U_{\alpha}$ and
$U_{\beta}$ on which
the section is represented as $w_{\alpha}=[f^k_{\alpha}]_1 v_{k,\alpha}$ and
$w_{\beta}=[f^k_{\beta}]_1 v_{k,\beta}$.
Now, we have that, since $\E$ is a subbundle of $\T_{S(1)}$
\[\tilde{v}_{k,\alpha}=[(h_{\alpha\beta})_k^h]_2\tilde{v}_{h,\beta},\]
which implies also that:
\[ [f^k_{\alpha}(h_{\alpha\beta})_k^h]_1=[f^h_{\beta}]_1.\]
We take two extensions $\tilde{w}_{\alpha}$ and $\tilde{w}_{\beta}$ on $U_{\alpha}$ and $U_{\beta}$
respectively: we
claim their difference lies in
$\mathcal{I}_S\T_{M,S(1)}^{\F}.$
We compute:
\begin{align*}
(\tilde{w}_{\beta}-\tilde{w}_{\alpha})|_S &=([\tilde{f}_{\alpha}^k]_2
\tilde{v}_{k,\alpha}-[\tilde{f}_{\beta}^h]_2 \tilde{v}_{h,\beta})|_S\\
&=([\tilde{f}_{\alpha}^k]_2 [h_{\alpha\beta,k}^h]_2\tilde{v}_{h,\beta}-[\tilde{f}^h_{\beta}]_2
\tilde{v}_{h,\beta})|_S\\
&=\big[[f_{\alpha}^k (h_{\alpha\beta})_k^h]_2 -[f_{\beta}^h]_2\big]_1 v_{h,\beta}=[0]_1.
\end{align*}
As stated, the difference between the two extensions lies in $\mathcal{I}_S\T_{M,S(1)}^{\F}$.
So, the map $\tilde{\pi}:\mathcal{E}|_S \to \mathcal{A}$ is an $\Ol_S$-morphism between $\E|_S$ and
$\A$ giving a splitting
of the following sequence:
\[\xymatrix{ 0\ar[r]
&\Hom(\mathcal{N}_S,\mathcal{N}_{\F,M})\ar[r]&\A_{\F,\E|_S}\ar[r]^{\Theta_1}&\E|_S\ar[r]&0 }.\]
where $\A_{\F,\E|_S}$ is the preimage of $\E|_S$ in $\A$ through $\Theta_1$.
\end{proof}
Therefore, recalling Section \ref{SecAtiyah} we
have that there is a partial
holomorphic connection on $\N_{\F,M}$ along $\E|_S$, given as follows:
\[\delta_v(s)=\tilde{X}_{\pi(\tilde{v})}(s),\]
where $\tilde{X}$ is the universal connection on $\A_{\N_{\F,M}}$.
\begin{Remark}
This connection may not be flat. Therefore we can use Bott Vanishing Theorem \ref{BottVanThe}
only in its non involutive form. But, in case $\E$ is involutive a stronger result holds.
\end{Remark}
\begin{Corollary}\label{Cor:InvolutiveSubsheaves}
Suppose $\E$ is an involutive coherent subsheaf of $\T_{S(1)}$ that, restricted to $S$, is a
subsheaf of $\F$ and $S$-faithful.
Then there exists a flat partial holomorphic connection $(\delta,\E)$ for $\N_{\F,M}$.
\end{Corollary}
\begin{proof}
From Proposition \ref{Prop:NonInvolutive} we already know there exists a partial holomorphic
connection along $\E$; since $\E$ is involutive we can check if it is flat:
\[\delta_u(\delta_v(s))-\delta_v(\delta_u(s))-\delta_{[u,v]}((s))=\pr([\tilde{u},[\tilde{v},\tilde{s
}]] - [\tilde{v},[\tilde{u},\tilde{s}]] - [[\tilde{u},\tilde{v}],\tilde{s}])=0,\]
by the Jacobi identity.
\end{proof}
\begin{Remark}
In the paper \cite{ABTHolMapFol} is defined the notion of Lie Algebroid morphism; given an
involutive
coherent subsheaf of $\T_{S(1)}$ the splitting that gives rise to the partial holomorphic
connection is a Lie algebroid morphism and last Corollary mirrors the fact that the universal
partial holomorphic connection is flat (Proposition \ref{PropHolConnection}).
\end{Remark}
\begin{Corollary}\label{Cor:TangentSheaf}
Suppose $\E$ is an involutive coherent subsheaf of $\T_{S(1)}$, whose restriction to $S$ is a
foliation of $S$ and is $S$-faithful.
Then there exists a flat partial holomorphic connection $(\delta,\E)$ for $\N_{S}$.
\end{Corollary}
\begin{proof}
If we take $\F=\T_S$ in Corollary \ref{Cor:InvolutiveSubsheaves} the assertion follows.
\end{proof}

\section{Singular holomorphic foliations of the first infinitesimal neighborhood}\label{Sec:Reduced}
This section is devoted to precise what we mean by case a singular foliations of infinitesimal
neighborhoods. In some sense, we want to prove an analogous of the following proposition, stated in
\cite{Suwa} and proved in \cite{MY}.
\begin{Proposition}\label{Prop:Reduced}
If a foliation is reduced, then $\codim S(\F)\geq 2$. If $\F$ is locally free and if $\codim
S(\F)\geq 2$, then $\F$ is reduced.
\end{Proposition}
We define now the main object of our treatment.
\begin{Definition}
A \textbf{singular foliation} of $S(k)$ is a rank $l$ coherent subsheaf $\F$ of $\T_{S(k)}$, such
that:
\begin{itemize}
\item for every $x\in S$ we have that $[\F_x,\F_x]\subseteq \F_x$ (where the bracket is the one
defined in Lemma
\ref{Lemma:Logaritmici});
\item the restriction of $\F$ to $S$, denoted by $\F|_S$, is a rank $l$ singular foliation of $S$.
\end{itemize}
\end{Definition}
\begin{Definition}
Let $\F$ be a singular holomorphic foliation of $S(k)$. We set $\N_{\F}=\T_{S(k)}/\F$ and we denote
by
$S(\F):=\Sing(\N_{\F})$ the \textbf{singular set of the foliation}\index{foliation of $k$-th
infinitesimal neighborhood!singular set}. 
\end{Definition}
\begin{Definition}
Let $\F$ be a singular foliation of $S(k)$. We say $\F$ is \textbf{reduced}
\index{foliation of $k$-th infinitesimal neighborhood!reduced} if it is full in $\T_{S(k)}$,
i.e., for any open set $U$ in $S$ we have that
\[\Gamma(U,\T_{S(k)})\cap\Gamma(U\setminus S(\F),\F)=\Gamma(U,\F).\]
\end{Definition}
\begin{Lemma}
Let $S$ be a submanifold of $M$ and let $\F$ be a singular foliation of $S(k)$; then there exists a
canonical way to associate to it a reduced singular foliation of $S(k)$.
\end{Lemma}
\begin{proof}
We cover now a neighborhood of $S$ by open sets $\{U_{\alpha}\}$ such that $\F_{U_{\alpha}\cap S}$
is generated by $v_{1,\alpha},\ldots, v_{l,\alpha}$ and on each $U_{\alpha}$ we can extend
the $v_{i,\alpha}$ to logarithmic vector fields $\tilde{v}_{i,\alpha}$ on $U_{\alpha}$.
On $U_{\alpha}$ the $\tilde{v}_{i,\alpha}$ define a distribution with sheaf of sections
$\D_{\alpha}$; please remark that this is a sheaf on $U_{\alpha}$, not on the whole $M$.
We define $\N_{\D_{\alpha}}=\T_M|_{U_{\alpha}}/\D_{\alpha}$ and denote by $S(\D_{\alpha})$ the set
of singularity of $\N_{\D_{\alpha}}$. 
In general, this distribution may not be reduced, i.e. $\Gamma(U_{\alpha},\T_M)\cap
\Gamma(U_{\alpha}\setminus S(D_{\alpha}),\D_{\alpha})\neq\Gamma(U_{\alpha},\D_{\alpha})$.
We take now the annihilator $(\D_{\alpha})^a=\{\omega\in\Omega_M\mid \omega(v)=0 \textrm{ for
every } v\in \D_{\alpha}\}$.
If we take its annihilator
$\tilde{\D}_{\alpha}:=((\D_{\alpha})^a)^a=\{w\in\T_M\mid \omega(w)=0 \textrm{ for
every } \omega\in (\D_{\alpha})^a\}$ we get now 
a reduced sheaf of sections of the distribution, generated by sections
$\tilde{w}_{1,\alpha},\ldots,\tilde{w}_{l,\alpha}$; we can take the same $l$ because, since we are
dealing with coherent sheaves, the rank is constant outside the singularity set.

Since $\Gamma(U_{\alpha},\D_{\alpha})\subset\Gamma(U_{\alpha},\tilde{\D}_{\alpha})$ we have that,
$\tilde{v}_{i,\alpha}=(h_{\alpha})_i^j\tilde{w}_{j,\alpha}$, where $(h_{\alpha})_i^j$ is a matrix
of holomorphic functions that \textbf{may} be singular on a subset of $U_{\alpha}$ of codimension
smaller than $2$, contained in $S(\D_{\alpha})$. Remark also that $S(\F)\subset S(\D_{\alpha})$ and
that the $\tilde{w}_{i,\alpha}$ are logarithmic vector fields. 

We want to check now that $\tilde{\D}_{\alpha}\otimes \Ol_{S(k)}|_{(U_{\alpha}\cap S)\setminus
S(\F)}$ generates $\F$ and is involutive.
We will denote the restriction of $\tilde{w}_{i,\alpha}$ to the $k$-th infinitesimal neighborhood
by $w_{i,\alpha}$.
Indeed, outside the singularity set, the matrix $(h_{\alpha})_i^j$ is invertible as a matrix of
holomorphic functions, with inverse $(g_{\alpha})^i_j$ which implies that the $w_{i,\alpha}$'s
generate $\F$.
We check the involutivity:
\begin{align*}
[ \tilde{w}_{i,\alpha},\tilde{w}_{i',\alpha}]\otimes[1]_{k+1}=&[(g_{\alpha})^j_i
\tilde{v}_{j,\alpha},(g_{\alpha})^{j'}_{i'}
\tilde{v}_{j',\alpha}]\otimes[1]_{k+1} \\
=&
[(g_{\alpha})^j_i]_{k+1}v_{j,\alpha}([(g_{\alpha})^{j'}_{i'})]_{k+1})v_{j',\alpha}\\&-[(g_{\alpha})^
{ j' } _{i'}]_{k+1}
v_{j'\alpha}([(g_{\alpha})^j_i]_{k+1})v_{j,\alpha}\\&+[(g_{\alpha})_{i'}^{j'}]_{k+1}[(g_{\alpha}
)^j_i ] _ { k+1 } [v_{j,\alpha},v_{j',\alpha}].
\end{align*}
Remark that $(g_{\alpha})^j_i$ is a matrix of meromorphic functions on $U_{\alpha}$ (this follows
from the Cramer rule for the inverse of a matrix),
and its inverse is a matrix of holomorphic functions.
Now, for each $v_{j,\alpha}$ we have that 
\[v_{j,\alpha}([(g_{\alpha})^{j'}_{i'}]_{k+1})=-[(g_{\alpha})_{i'}^{j''}]_{k+1}
v_{j,\alpha}([(h_{\alpha})^{i''}_{j''}]_{k+1})[(g_{\alpha})^{j'}_{i''}]_{k+1},\]
and therefore:
\begin{align*}
[(g_{\alpha})^j_i]_{k+1}v_{j,\alpha}&([(g_{\alpha})^{j'}_{i'})]_{k+1})v_{j',\alpha}\\
=&-[(g_{\alpha})^j_i]_{k+1}[(g_{\alpha})_{i'}^{j''}]_{k+1}
v_{j,\alpha}([(h_{\alpha})^{i''}_{j''}]_{k+1})[(g_{\alpha})^{j'}_{i''}]_{k+1}v_{j',\alpha}\\
=&-[(g_{\alpha})_{i'}^{j''}]_{k+1} w_{i,\alpha}([(h_{\alpha})^{i''}_{j''}]_{k+1})w_{i'',\alpha}.
\end{align*}
A similar reasoning holds for the second summand in the involutivity check. If we denote by
$[a_{j,j'}^{j''}]_{k+1}$ the elements of $\Ol_{S(k)}$ such that 
\[ [v_{j,\alpha},v_{j',\alpha}]=[a_{j,j'}^{j''}]_{k+1} v_{j'',\alpha}\]
we have that:
\begin{align*} 
[(g_{\alpha})_{i'}^{j'}]_{k+1}[(g_{\alpha})^j_i]_{ k+1 }& [v_{j,\alpha},v_{j',\alpha}]\\
=&-[(g_{\alpha})_{i'}^{j'}]_{k+1}[(g_{\alpha})^j_i]_{ k+1 } [a_{j,j'}^{j''}]_{k+1}v_{j'',\alpha}\\
=&-[(g_{\alpha})_{i'}^{j'}]_{k+1}[(g_{\alpha})^j_i]_{ k+1 }
[a_{j,j'}^{j''}]_{k+1}[(h_{\alpha})_{j''}^{i''}]_{k+1} w_{i'',\alpha}.
\end{align*}
The point these computations prove is that
$[\tilde{w}_{i,\alpha},\tilde{w}_{i',\alpha}]\otimes[1]_{k+1}$ belongs to the module generated by
the $w_{i,\alpha}$'s over the meromorphic functions. But, a priori, we know that this bracket is a
holomorphic section of $\T_{S(k)}$ and therefore it belongs to the $\Ol_{S(k)}$-module generated by
the $w_{i,\alpha}$'s.
\end{proof}
\begin{Remark}\label{Rem:Codimension}
By the proof of the Lemma above and by \ref{Prop:Reduced} we have an important consequence: each
one of the extensions $\tilde{w}_{i,\alpha}$ has a singularity set of codimension at least $2$.
\end{Remark}

\section{Localization of Chern classes}\label{SectionVanishing}
In this section we give a short account of the theory that permits us to compute residues.
Thanks to Theorem 6.1 in \cite{ABTHolMapFol}, an extended version of Bott's vanishing theorem, we know
that the existence of a partial holomorphic connection along a non involutive subbundle gives rise to vanishing of the
characteristic classes of the bundle endowed with the connection.
\begin{Theorem}[\cite{ABTHolMapFol}, Thm. 6.1]\label{BottVanThe}
Let $S$ be a complex manifold, $F$ a sub-bundle of $TS$ of rank $l$ , and $E$
a complex vector bundle on $S$. Assume we have a partial holomorphic connection on $E$
along $F$ . Then:
\begin{itemize}
 \item every symmetric polynomial in the Chern classes of $E$ of degree larger than $\dim S-
         l+\lfloor l/2\rfloor$ vanishes (where $\lfloor x\rfloor$ is the largest integer less than or equal to $x$).
\item Furthermore, if $F$ is involutive and the partial holomorphic connection is flat then
every symmetric polynomial in the Chern classes of E of degree larger than $\dim S- l$
  vanishes.
\end{itemize} 
\end{Theorem}
We give a sketch of the localization process we use, following \cite[pag. 194]{Suwa}.
Indeed, let $M$ be a complex manifold and let $S$ be a compact complex submanifold of dimension $n$.
Suppose we have an involutive coherent subsheaf of $\T_{S(1)}$ of rank $l$. Let $\Sigma:=\Sing(\T_{S(1)}/\F)$
and suppose $\F|_{S\setminus\Sigma}$ is a foliation of $S$ of rank $l$. Then, we have a foliation of the first
infinitesimal neighborhood outside an algebraic subset $\Sigma$. From our discussion in Section \ref{SecAtiyah} we know
we have a flat partial holomorphic connection on $\N_{\F,M}$ along $\F|_S$ on $S\setminus\Sigma$;
therefore the characteristic classes obtained evaluating a symmetric polynomial of degree $n-k$ larger than $n-l$ on the
Chern classes of $\N_{\F,M}$ vanish when restricted to $S\setminus\Sigma$. Let now $\eta$ be
such a characteristic class, obtained from a symmetric
polynomial of degree $n-k$; if we inspect the long exact sequence for the cohomology of the pair $(S,S\setminus \Sigma)$
\[\xymatrix{\ldots \ar[r]
&H^{2(n-k)}(S,S\setminus\Sigma) \ar[r] &H^{2(n-k)}(S) \ar[r] &H^{2(n-k)}(S\setminus\Sigma) \ldots},\]
we can notice that, since $\eta$ vanishes on $S\setminus\Sigma$ we can lift it to a class in
$H^{2(n-k)}(S,S\setminus\Sigma)$; please note this lift depends on the partial holomorphic
connection and therefore we denote it by $(\eta,\nabla)$. Since $S$ is compact we now can apply
Poincar\'{e} duality $P_S$ and Alexander duality
$A_{\Sigma}$ obtaining the following commutative diagram:
\[\xymatrix{&H^{2(n-k)}(S,S\setminus\Sigma) \ar[r]\ar[d]^{A_{\Sigma}} &H^{2(n-k)}(S)\ar[d]^{P_S}\\
&H_{2k}(\Sigma)\ar[r]^\iota &H_{2k}(S).}\]
Since $\Sigma$ is the union of its connected components $\Sigma_{\alpha}$, its homology is the direct sum of the
homologies of each connected component; the Alexander duality is compatible with this decomposition.
Therefore we obtain the following residue formula:
\[\sum_{\alpha} \iota_{\alpha}(A_{\Sigma_{\alpha}})((\eta,\nabla))=P_S(\eta).\]
In case $k=0$ what we are doing is to associate to each connected component $\Sigma_{\alpha}$ a number and the sum
of all these numbers is $P_S(\eta)$.
We want now to understand how we compute the residue of a characteristic class arising from a
polynomial of degree $n$; we will use the tool of \u{C}ech-deRham cohomology (refer to
\cite{Suwa}), before proceeding we need to cite some results and define some notation.
\begin{Remark}
Please note that the Bott Vanishing Theorem in this form not only tells us that the de Rham
cohomology class of a Chern
class vanishes, but also that the form representing it vanishes.
\end{Remark}
We refer to \cite[pag. 71]{Suwa} for the theory regarding virtual bundles. Before progressing we
need a couple of definitions.
\begin{Definition}
A \textbf{virtual bundle} is an element in the $K$-group $K(M)$ of
$M$ \cite{Milnor},\cite{Suwa}.
In particular, if we have complex vector bundles $E_i$ with $i=0,\ldots,q$ over a smooth manifold
$M$ we may consider the virtual bundle $\xi=\sum_{i=0}^q (-1)^i E_i$.
\end{Definition}
\begin{Definition}\label{DefinitionCompatibleWithSequence}
Let
\begin{equation}\label{virtualbundles}
\xymatrix{&0\ar[r] &E_q\ar[r]^{\psi_q}\ar[r]&\ldots\ar[r]&E_1\ar[r]^{\psi_1} &E_0\to 0,}
\end{equation}
be a sequence of vector bundles on $M$, and for each $i=0,\ldots, q$, let $\nabla^{(i)}$ be a
connection for $E_i$. We
say that the family $\nabla^{(q)},\ldots,\nabla^{(0)}$ is \textbf{compatible with the sequence} if,
for each $i$, the
following diagram commutes:
\[\xymatrix{ 
A^0(M,E_i)     \ar[r]^{\nabla^{(i)}}\ar[d]_{\psi_i} &A^1(M,E_i)\ar[d]^{1\otimes \psi_i} \\
A^0(M,E_{i-1}) \ar[r]^{\nabla^{(i-1)}} &A^1(M,E_{i-1}).}\]
\end{Definition}
\begin{Remark}
The following Proposition from the book of Suwa is really general in scope, dealing with virtual
bundles and
\u{C}ech-de Rham cohomology.

We refer to \cite[p. 164]{Milnor} for the proof of the formulas which express the Chern classes of
$E\oplus F$ as products and sums of the Chern classes of $E$ and $F$.
In general for a direct sum of vector bundles $E\oplus F$, the total Chern
class $\textrm{c}(E\oplus F)=\textrm{c}(E)\smile\textrm{c}(F)$.
This permits us to compute the Chern classes of $E\oplus F$ as polynomials in the Chern classes of
$E$ and $F$; the simplest example is the first Chern class of the bundle $E\oplus F$: we have that
$c_1(E\oplus F)=c_1(E)+c_1(F)$.

A really interesting fact about virtual bundles is that Chern classes behave naturally with respect
to them, generalizing the discussion about the Chern class of a direct sum above. In general, if
$\phi$ is a symmetric polynomial, $\xi:=\sum_{i=0}^q E_i$ is a virtual bundle
and ${\nabla}^{\bullet}$ denotes a family of connections $\nabla^{(q)},\ldots,\nabla^{(1)}$ for
$\xi$ then we can express $\phi(\xi)$ as a finite sum
\[\phi(\xi)=\sum_l\phi_l^{(0)}(E_0)\wedge\ldots\wedge\phi_l^{(q)}(E_q),\]
where $\phi_l^{(i)}(E_i)$ are polynomials in the Chern classes of $E_i$ for each $i$ and $l$.
Then
\[\phi({\nabla}^{\bullet})=\sum_l
\phi_l^{(0)}(\nabla^{(0)})\wedge\ldots\wedge\phi_l^{(q)}(\nabla^{(q)})\]
represents the cohomology class of $\phi(\xi)$.

In \cite{Suwa} it is proved that there exists a form
$\phi({\nabla}^{\bullet}_0,\ldots,{\nabla}^{\bullet}_p)$, called the \textbf{Bott difference
form}
such that 
\[\sum_{\nu=0}^p\phi({\nabla}^{\bullet}_0,\ldots,\widehat{\nabla}^{\bullet}_\nu,\ldots,{\nabla}^{
\bullet } _p)+(-1)^p d\phi({ \nabla } ^ { \bullet}_0,\ldots,{\nabla}^{\bullet}_p)=0,\]
where the hat means the hatted family is not taken into consideration.
\end{Remark}
\begin{Proposition}[\cite{Suwa}, p. 73]\label{Prop:CompatibleSequence}
Suppose sequence \eqref{virtualbundles} is exact. Let $\phi$ be a symmetric polynomial and
$\nabla^{\bullet}_k=\{\nabla^{(q)}_k,\ldots,\nabla^{(0)}_k\}$ for $k=0,\ldots, p$ families of
connections compatible with
\eqref{virtualbundles} for the virtual bundle
$\tilde{\xi}=\sum_{i=1}^q(-1)^{i-1}E_i$.
Then:
\[\phi(\tilde{\nabla}^{\bullet}_0,\ldots,\tilde{\nabla}^{\bullet}_p)=\phi(\nabla^{(0)}_0,\ldots,
\nabla^{(0)}
_p)\]
Similarly for other ``partitions'' of the virtual bundle $\xi$. In particular
\[\phi(\tilde{\xi})=\phi(E_0),\]
\end{Proposition}
We state now Bott vanishing theorem in the version for virtual bundles.
\begin{Theorem}[Bott Vanishing theorem, \cite{Suwa} pag.
76]\label{BottVanishingTheoremforVirtualBundles}
Let $M$ be a complex manifold of dimension $n$ and $F$ an involutive subbundle of rank $p$ of $TM$.
Also, for each
$i=0,\ldots, q$ let $E_i$ be a bundle and let $\nabla^{(i)}_1,\ldots,\nabla^{(i)}_k$ be partial
holomorphic connections for $E_i$ along $F$, then, for any
homogeneous
symmetric polynomial $\phi$ of degree $d> n-p$ we have
\[\phi(\nabla^{\bullet}_1,\ldots,\nabla^{\bullet}_k)\equiv 0,\]
where $\nabla^{\bullet}_j=(\nabla^{(q)}_j,\ldots,\nabla^{(0)}_j)$, for $j=1,\ldots,k$.
\end{Theorem}

\section{Index theorems for foliations and involutive closures}\label{IndFolia}
Following the work \cite{Suwa} and the articles \cite{ABTHolMap}, \cite{ABTHolMapFol}, we know that
the
existence of a
partial holomorphic connection, thanks to Bott's Vanishing Theorem \ref{BottVanThe}, gives rise to
the vanishing of some
of the Chern classes of a vector bundle and therefore to an index theorem.
In Section \ref{SecAtiyah} we found a concrete realization of the Atiyah sheaf for the
normal bundle
of a foliation as
a quotient of the ambient tangent bundle and we proved that
the Atiyah sequence splits if there exists a foliation of the first infinitesimal neighborhood.
In this section we state the index theorems that follow directly from our treatment.

The simpler case is when we have a foliation of the first infinitesimal neighborhood; then we have a
partial holomorphic
connection on $\N_{\F,M}$ (Corollary \ref{CorollaryHolomorphicAction}) and so, Bott's Vanishing
Theorem (Theorem \ref{BottVanThe}) permits us to prove the following.
\begin{Theorem}\label{Theo:IndexFoliations}
Let $S$ be a codimension $m$ compact submanifold of a $n$ dimensional complex manifold $M$. Let $\F$
be a rank $l$
foliation on $S$, such that it extends to the first infinitesimal neighborhood of $S\setminus
S(\F)$, and
let $S(\F)=\bigcup_{\lambda}\Sigma_{\lambda}$ be the decomposition of $S(\F)$ in
connected components.
Then for every symmetric homogeneous polynomial $\phi$ of degree $k$ larger than $n-m-l$ we can
define the residue
$\Res_{\phi}(\F,\N_{\F,M};\Sigma_{\lambda})\in H_{2(n-m-k)}(\Sigma_{\alpha})$ depending only on the
local behaviour of
$\F$ and $\N_{\F,M}$ near $\Sigma_{\lambda}$ such that:
\[\sum_{\lambda} \Res_{\phi}(\F,\N_{\F,M};\Sigma_{\lambda})=\int_S \phi(\N_{\F,M}),\]
where $\phi(\N_{\F,M})$ is the evaluation of $\phi$ on the Chern classes of $\N_{\F,M}$.
\end{Theorem}
\begin{proof}
If we denote by $F$ the vector bundle associated to $\F$ we have that the virtual bundle associated
with the sheaf $\N_{\F,M}$ is nothing else but $[TM|_S-F|_S]$.
Now, outside the singularity set of $\F$, this virtual bundle is a vector bundle on $S$ and by
Corollary \ref{CorollaryHolomorphicAction} it admits a partial holomorphic connection along $\F|_S$;
Bott Vanishing Theorem tells us that for each $\phi$ of degree $k$ larger than $n-m-l$ we have that
the restriction of $\phi(\N_{\F,M})$ to $S\setminus S(\F)$ is represented by the $0$ form.
Applying the localization process as in Section \ref{SectionVanishing} the result follows.
\end{proof}
Using now the results of Section \ref{SectionNonInvolutive} we can prove a stronger result.
\begin{Theorem}
Let $S$ be a codimension $m$ compact submanifold of a $n$ dimensional complex manifold $M$. Let $\F$
be a foliation on $S$ and let $\E$ be a rank $l$ subsheaf of $\T_{S(1)}$ that, restricted to
$S$, is a subsheaf of $\F$. Suppose moreover that it is $S$-faithful.
Let $\Sigma=S(\F)\cup S(\E)$ and let $\Sigma=\bigcup_{\lambda}\Sigma_{\lambda}$ be the decomposition
of $\Sigma$ in
connected components.
Then for every symmetric homogeneous polynomial $\phi$ of degree $k$ larger than
$n-m-l+\lfloor l/2 \rfloor$ we can
define the residue
$\Res_{\phi}(\E,\N_{\F,M};\Sigma_{\lambda})\in H_{2(n-m-k)}(\Sigma_{\alpha})$ depending only on the
local behaviour of
$\F$ and $\N_{\F,M}$ near $\Sigma_{\lambda}$ such that:
\[\sum_{\lambda} \Res_{\phi}(\E,\N_{\F,M};\Sigma_{\lambda})=\int_S \phi(\N_{\F,M}),\]
where $\phi(\N_{\F,M})$ is the evaluation of $\phi$ on the Chern classes of $\N_{\F,M}$.
\end{Theorem}
\begin{Remark}
Please note that by Corollary \ref{Cor:InvolutiveSubsheaves} if $\E$ is involutive the above
holds with $n-m-l$ instead of $n-m-l+\lfloor l/2 \rfloor$.
\end{Remark}
Suppose now we have a foliation of $M$, transversal to $S$, $2$-splitting submanifold of $M$, and
suppose we have a
first order
$\F$-faithful splitting $\sigma_2$ outside an algebraic subset. Now, the foliation $\F^{\sigma_2}$
is a foliation of the first infinitesimal neighborhood of $S$ and by
Theorem \ref{Theo:IndexFoliations} we have the following.
\begin{Theorem}
Let $S$ be a codimension $m$ $2$-splitting compact submanifold of a $n$ dimensional complex
manifold $M$. Let $\F$
be a rank $l$ holomorphic foliation defined on a neighborhood of $S$. Suppose there is a
$2$-splitting first order
$\F$-faithful outside an analytic subset $\Sigma$ of $U$ containing $S(\F)\cap S$ and that $S$
is not contained in
$\Sigma$. Let
$\Sigma=\bigcup_{\lambda}\Sigma_{\lambda}$ be the decomposition of $\Sigma$ in connected components.
Then for every symmetric homogeneous polynomial $\phi$ of degree $k$ bigger than $n-m-l$ we can
define the residue 
$\Res_{\phi}(\F,\N_{\F^{\sigma},M};\Sigma_{\lambda})\in H_{2(n-m-k)}(\Sigma_{\alpha})$ depending
only on the
local behaviour of
$\F$ and $\N_{\F^{\sigma},M}$ near $\Sigma_{\lambda}$ such that:
\[\sum_{\lambda} \Res_{\phi}(\F,\N_{\F^{\sigma},M};\Sigma_{\lambda})=\int_S
\phi(\N_{\F^{\sigma},M}),\]
where $\phi(\N_{\F^{\sigma},M})$ is the evaluation of $\phi$ on the Chern classes of
$\N_{\F^{\sigma},M}$.
\end{Theorem}
\begin{Remark}\label{Rem:Application}
An interesting research path is to investigate the relation between $\N_{\F^{\sigma},M}$ and
$\N_{\F}|_S$.
The motivation behind this question is easily seen: suppose $M$ is a complex surface and $\F$ is a
dimension $1$ singular foliation transversal to $S$, a $2$-splitting $1$ dimensional  submanifold.
Suppose moreover that the sequence
\[\xymatrix{0\ar[r]& \F\ar[r]&\T_M\ar[r]&\N_{\F}\ar[r]& 0}\]
splits when restricted to $S$.
Suppose we have a first order $\F$-faithful splitting $\sigma^*$ outside $\Sigma$; thanks
to the splitting of $S$ and the splitting of the sequence above $\sigma^*$ induces an isomorphism
between $\N_{\F}|_S$ and $\N_{\F^{\sigma},M}$ outside the singular points of $\F$.
Suppose $\F$ admits an algebraic compact leaf $L$.
If we denote by $\N_L$ the normal sheaf to this leaf we have that $\N_L\equiv\N_{\F}|_L$ and we have
that 
\[\int_S c_1(N_{\F^{\sigma,M}})=\int_S c_1(N_{\F}|_S)=\int_S c_1(N_L)=(L\cdot S)\]
is the intersection number between $L$ and $S$.
Therefore we could apply this test to foliations, getting informations on the intersection numbers
of possible analytic leaves.  
\end{Remark}
The other results follow from the splitting of the sequence \eqref{folest} studied in Section
\ref{SecExtension}. In case $\F$ has rank $1$ and we do not need to take care of involutivity we
have the following consequence of \ref{Theo:IndexFoliations}.
\begin{Theorem}
Let $S$ be a codimension $m$ compact submanifold splitting in an $n$ dimensional complex manifold
$M$, and suppose $\F$
is
a rank $1$ holomorphic foliation defined on $S$. Suppose sequence \eqref{folest} splits and let
$\Sigma=S(\F)$ and
let $\Sigma=\bigcup_{\lambda}\Sigma_{\lambda}$ be the decomposition
of $\Sigma$ in connected components.
Then for every symmetric homogeneous polynomial $\phi$ of degree $n-m$ we can define the residue
$\Res_{\phi}(\F,\N_{\F,M};\Sigma_{\lambda})\in H_{0}(\Sigma_{\alpha})$ depending only on the local
behaviour of
$\F$ and
$\N_{\F,M}$ near $\Sigma_{\lambda}$ such that:
\[\sum_{\lambda} \Res_{\phi}(\F,\N_{\F,M};\Sigma_{\lambda})=\int_S \phi(\N_{\F,M}),\]
where $\phi(\N_{\F,M})$ is the evaluation of $\phi$ on the Chern classes of $\N_{\F,M}$.
\end{Theorem}
Now, if $\F$ has rank $l$ and we suppose its extension arising from the splitting of \eqref{folest}
is involutive we have the following. 
\begin{Theorem}
Let $S$ be a codimension $m$ compact submanifold splitting in $M$, $n$ dimensional complex manifold,
and suppose $\F$ is
a rank $l$ holomorphic foliation defined on $S$. Suppose sequence \eqref{folest} splits and that the
image of
$\tilde{\F}/\mathcal{V}$ in $\tilde{\F}$ is involutive. Let $\Sigma=S(\F)$ and
let $\Sigma=\bigcup_{\lambda}\Sigma_{\lambda}$ be the decomposition
of $\Sigma$ in connected components.
Then for every symmetric homogeneous polynomial $\phi$ of degree $k$ larger than $n-m-l$ we can
define the residue
$\Res_{\phi}(\F,\N_{\F,M};\Sigma_{\lambda})\in H_{2(n-m-k)}(\Sigma_{\alpha})$ depending only on the
local behaviour of
$\F$ and
$\N_{\F,M}$ near $\Sigma_{\lambda}$ such that:
\[\sum_{\lambda} \Res_{\phi}(\F,\N_{\F,M};\Sigma_{\lambda})=\int_S \phi(\N_{\F,M}),\]
where $\phi(\N_{\F,M})$ is the evaluation of $\phi$ on the Chern classes of $\N_{\F,M}$.
\end{Theorem}
In case we drop the involutivity assumption we have a weaker form thanks to Theorem
\ref{Prop:NonInvolutive}.
\begin{Theorem}
Let $S$ be a codimension $m$ compact submanifold splitting in $M$ complex manifold of dimension $n$.
Suppose
$\F$ is a foliation of $S$ of rank $l$ and suppose sequence \eqref{folest} splits.
Let $\Sigma=S(\F)$ and let $\Sigma=\bigcup_{\lambda}\Sigma_{\lambda}$ be the decomposition of its
singular set in
connected
components. Then, for every symmetric homogeneous polynomial $\phi$ of degree $k$ larger than
$n-m-l+\lfloor l/2\rfloor$
we can define the residue $\Res_{\phi}(\F,\N_{\F,M};\Sigma_{\alpha})\in
H_{2(n-m-k)}(\Sigma_{\alpha})$ depending only on
the local behaviour of $\F$ and $\N_{\F,M}$ near $\Sigma_{\alpha}$ such that:
\[\sum_{\lambda}\Res_{\phi}(\F,\N_{\F,M};\Sigma_{\alpha})=\int_S \phi(\N_{\F,M}),\]
where $\phi(\N_{\F,M})$ is the evaluation of $\phi$ on the Chern classes of $\N_{\F,M}$.
\end{Theorem}
In the case $S$ has first order extendable tangent bundle the vanishing of the cohomology class
associated to \eqref{folest} 
follows directly from Corollary \ref{FirstOrder}, but we cannot say anything about the involutivity
of this extension.
\begin{Theorem}
Let $S$ be a codimension $m$ compact submanifold splitting in an $n$ dimensional complex manifold
$M$, and with first
order extendable tangent bundle.
Let $\F$ be a rank $l$ holomorphic foliation defined on $S$. Let $\Sigma=S(\F)$ and let
$\Sigma=\bigcup_{\lambda}\Sigma_{\lambda}$ be the decomposition
of $\Sigma$ in connected components.
Then for every symmetric homogeneous polynomial $\phi$ of degree $k$ larger than
$n-m-l+\lfloor l/2 \rfloor$ we can
define the residue
$\Res_{\phi}(\F,\N_{\F,M};\Sigma_{\lambda})\in H_{2(n-m-k)}(\Sigma_{\alpha})$ depending only on the
local behaviour of
$\F$ and
$\N_{\F,M}$ near $\Sigma_{\lambda}$ such that:
\[\sum_{\lambda} \Res_{\phi}(\F,\N_{\F,M};\Sigma_{\lambda})=\int_S \phi(\N_{\F,M}),\]
where $\phi(\N_{\F,M})$ is the evaluation of $\phi$ on the Chern classes of $\N_{\F,M}$.
\end{Theorem}
\begin{Remark}
From the theory developed in Section \ref{SecExtension} it seems likely that, given a foliation
$\F$ of the first infinitesimal neighborhood and an involutive subsheaf $\G$ of rank $l$  of
$\F|_S$ this subsheaf extends to a subsheaf of $\F$, possibly non involutive. This does not give
rise to new index theorems, but is indeed worth noting and investigating. 
\end{Remark}
Another interesting result following from our theory is obtained by defining, for a coherent
subsheaf $\E$ of $\T_{S(1)}$, a natural object,
its involutive closure, the smallest involutive subsheaf containing $\E$.
Thanks to the machinery developed in Section
\ref{SectionNonInvolutive}, it is proved that the existence of $\E$ gives rise to vanishing theorems
for
its involutive closure.
\begin{Definition}
Let $\E$ be a coherent subsheaf of $\T_{S(1)}$ such that $\E|_S$ is non empty. We denote by
$\Sing(\E)$ the set $\{x\in S\mid
\T_{S(1)}/\E \textrm{is
not} \Ol_{S(1),x}-\textrm{free}\}.$ On $S\setminus \Sing(\E)$ we define the \textbf{involutive
closure} $\G$ of $\E$ in $S$ to
be the intersection of all the coherent involutive subsheaves of $\T_{S}$ containing $\E|_S$.
\end{Definition}
Recall that the intersection of coherent subsheaves of $\T_S$ is again a coherent subsheaf of
$\T_S$; now, $\G$ is involutive by definition and therefore gives rise to a foliation of $S$.
Clearly, $\E|_S$ is a subsheaf of $\G$ and we can apply Proposition \ref{Prop:NonInvolutive}
getting the following result.
\begin{Theorem}\label{Theorem:InvolutiveClosure}
Let $S$ be a codimension $m$ compact submanifold of $M$ complex manifold of dimension $n$. Suppose
$\E$ is a coherent
subsheaf of $\T_{S(1)}$ of rank $l$, whose
restriction $\E|_S$ has rank $l$. Let $\G$ be the involutive closure of $\E$ in $S$. Let
$\Sigma=S(\E)\cup S(\G)=\bigcup_{\alpha}\Sigma_{\alpha}$ be the decomposition of $\Sigma$
in
connected
components. Then, for every symmetric homogeneous polynomial $\phi$ of degree $k$ larger than
$n-m-l+\lfloor l/2\rfloor$
we can define the residue $\Res_{\phi}(\E|_S,\N_{\G,M};\Sigma_{\alpha})\in
H_{2(n-m-k)}(\Sigma_{\alpha})$ depending only
on
the local behaviour of $\E|_S$ and $\N_{\G,M}$ near $\Sigma_{\alpha}$ such that:
\[\sum_{\lambda}\Res_{\phi}(\E|_S,\N_{\G,M};\Sigma_{\alpha})=\int_S \phi(\N_{\G,M}),\]
where $\phi(\N_{\G,M})$ is the evaluation of $\phi$ on the Chern classes of $\N_{\G,M}$.
\end{Theorem}

\section{Computing the residue in the simplest case}\label{SecCompResTang}
In this section we will compute the residue for a codimension $1$ foliation of the first
infinitesimal neighborhood of a codimension $1$ submanifold in a surface.
Let $(U_1,x,y)$ be a neighborhood of $0$ in $\mathbb{C}^2$, let $S=\{x=0\}$; let $\F$ be a foliation
of $S(1)$ such that $\Sing(\F)=\{0\}$ and let $v$ be a generator of $\F$; that is a holomorphic
section of
$\T_{S(1)}$ with an isolated singularity in $0$.
Supposing $\F$ reduced, from Section \ref{Sec:Reduced} and Remark \ref{Rem:Codimension}
we see that this is assumption does not give rise to a loss of generality for our computation.
\begin{Remark}\label{Rem:Computation}
Please note also that, if we denote by $\tilde{v}$ an extension of $v$ to $U_1$ and by
$\tilde{\F}$ the foliation generated by it, thanks to how we defined the holomorphic action and the
theory developed for local extensions, the computation of this residue could be reduced to the
computation of the residue given by the Lehmann-Khanedani-Suwa action of $\tilde{v}$ on
$\N_{\F}|_S$, which can be found e.g. in \cite[Ch. IV, Theorem 5.3]{Suwa}.
\end{Remark}
We will, anyway, compute the index explictly in the framework we developed.
Call $U_0:=U_1\setminus\{0\}$; with an abuse of notation we will also say $M:=U_1$.
Let $G$ be the trivial line bundle on $S$; we can see $v|_S$ as a holomorphic homomorphism
between $G$ and $TS$.
On $U_0$ we can see $G$ as a subbundle of $TM|_S$, moreover $G$ embedded through $v|_S$ is
nothing else that the bundle associated to $\F|_S$. Therefore, we can speak of the virtual bundle
$[TM|_S-G]$, which coincides, on $U_0$, with the normal bundle to the
foliation $\F|_{U_0\cap S}$ in the ambient tangent bundle $TM|_{U_0\cap S}$, denoted
by $N_{\F,M}$.
Since the only homogeneous symmetric polynomial in dimension $1$ is the trace we would like to
compute the residue for the first Chern class of $[TM|_S-G]$, whose sheaf of sections is
$\N_{\F,M}$. Being the first Chern class additive, we are going to compute $c_1(TM|_S)-c_1(G)$.
If $U_0$ is small enough, thanks to the embedding of $G$ into $TM|_S$ we have that on $U_0$ we can
see $TM|_S$ as the
direct sum $G\oplus N_{\F,M}$. 
We are going to apply Proposition \ref{Prop:CompatibleSequence} to the following sequence:
\[\xymatrix{0\ar[r] &F|_{U_0\cap S}\ar[r] &TM|_{U_0\cap S}\ar[r] & N_{\F,M}\ar[r] &0.}\]
We want to build on $U_0$ a family of connections compatible with the sequence, so that Theorem
\ref{BottVanishingTheoremforVirtualBundles} implies that $c_1(N_{\F,M})$ on $U_0$ is $0$.
We proved that the existence of a foliation of the first infinitesimal neighborhood gives rise to
partial connection on $N_{\F,M}$.
Now, thanks to Corollary \ref{CorollaryHolomorphicAction} we can compute the actual connection
matrix of this partial holomorphic connection on $\N_{\F,M}$ and extend
it to a connection on $N_{\F,M}$, denoted by $\nabla$. 
To build a family of connections simplifying our computations we take on $U_0\cap S$ the connection
$\nabla_0^G$ which is trivial with respect with the generator $1_G$ of the trivial line bundle $G$. 
Since $TM|_S$ on $U_0\cap S$ is the direct
sum of $G$ and $N_{\F,M}$ we let the connection for $TM|_S$ be the direct sum connection 
$\nabla_0^{TM}:=\nabla\oplus\nabla_0^G.$
Both $\nabla_0^{TM}$ and $\nabla_0^G$ are holomorphic connections along $F$, therefore we
can apply Bott's
Vanishing in the version for virtual bundles and obtain that $c_1(N_{\F,M})\equiv 0$ on $U_0$.

In \u{C}ech-de Rham cohomology relative to the cover $\{U_0,U_1\}$ the first Chern class of
$\N_{\F,M}$ is represented
as a triple $(\omega_0,\omega_1,\sigma_{01})$, where $\omega_0$ is the first Chern class of
$\N_{\F,M}$ on $U_0$,
$\omega_1$ is the first Chern class of $\N_{\F,M}$ on $U_1$ while $\sigma_{01}$ is a $1$-form, the
Bott
difference form, i.e., a $1$-form such that $\omega_1-\omega_0=d\sigma_{01}$ on $U_0\cap U_1$(for
a complete treatment, refer to \cite{Suwa}).
Due to the additivity of the first Chern class, to compute the first Chern class of
$\N_{\F,M}$ we need to
compute the first Chern classes of $G$ and $TM|_S$ on $U_1$ (we already know the first Chern class
of
$\N_{\F,M}$ on $U_0$ is $0$) and the Bott difference forms $c_1(\nabla^{TM}_0,\nabla^{TM}_1)$ and
$c_1(\nabla^{G}_0,\nabla^{G}_1)$. On
$U_1$ we can take, again, as a connection for $G$ the connection which is trivial with respect to
the
generator $1_G$ of $G$:
therefore $c_1(\nabla^{G}_0,\nabla^{G}_1)=0$, since the connections for $G$ on $U_0$ and $U_1$ are
the same.
On $U_1$ we take as $\nabla^{TM}_1$ the $\partial/\partial x$, $\partial/\partial
y$ trivial
connection; then $c_1(\nabla^{TM}_1)=0$ and the problem reduces to compute the Bott difference form
$c_1(\nabla^{TM}_0,\nabla^{TM}_1)$.
To compute it we need the connection matrix for $\nabla_0^{TM}$ with respect to the frame
$\partial/\partial x$, $\partial/\partial y$. First of all we compute the action of $\nabla$ on the
equivalence class
$\nu=[\partial/\partial
x]$ in $N_{\F,M}$.
The generator $v$ of $\F$ is written in coordinates as
\[[A]_2\frac{\partial}{\partial x}+[B]_2\frac{\partial}{\partial y},\]
where $[A]_2$ belongs to $\mathcal{I}_S/\mathcal{I}_S^2$. In the following, we shall denote by $v_S$
the restriction of
$\tilde{v}$ to $S$; in coordinates we have that $v_S=[B]_1\partial/\partial y$.
We compute now the action of $F$ on $\N_{\F,M}$, recalling Corollary
\ref{CorollaryHolomorphicAction}
\[\nabla_{v_S}(\nu)=\pr\bigg(\bigg[[A]_2\frac{\partial}{\partial x}+[B]_2\frac{\partial}{\partial
y},\frac{\partial}{\partial x}\bigg]\bigg|_S\bigg)=-\bigg[\frac{\partial A}{\partial
x}\bigg]_1\nu.\]
We compute now the connection matrix for $\nabla$. Since
\[-\bigg[\frac{\partial A}{\partial x}\bigg]_1=([C]_1\cdot dx+[D]_1\cdot
dy)([B]_1\frac{\partial}{\partial
y})=[D\cdot B]_1,\]
it follows that the connection matrix is nothing else but:
\[\omega=-\bigg[\frac{\partial A}{\partial x}\frac{1}{B}\bigg]_1 dy.\]
We have now all the tools needed to compute the connection matrix for $\nabla_0^{TM}$:
\begin{align*}
\nabla_0^{TM}\bigg(\frac{\partial}{\partial x}\bigg)&=\nabla(\nu)=-\bigg[\frac{\partial A}{\partial
x}\cdot\frac{1}{B}\bigg]_1 dy\otimes\frac{\partial}{\partial x},\\
\nabla_0^{TM}\bigg(\frac{\partial}{\partial y}\bigg)&=\nabla_0^G\bigg(\frac{1}{B}\cdot
v\bigg)=-\bigg[\frac{dB}{B^2}\bigg]_1\cdot
v=-\bigg[\frac{dB}{B}\bigg]_1\otimes\frac{\partial}{\partial y}.
\end{align*}
Thus the connection matrix has the following form:
\[
\begin{bmatrix}
-\bigg[\frac{\partial A}{\partial x}\frac{1}{B}\bigg]_1 dy & 0  \\
0 & -\bigg[\frac{dB}{B}\bigg]_1  \\ 
\end{bmatrix}.
\]
We can compute now the Bott difference form. We consider the bundle $TM\times [0,1]\to M\times
[0,1]$
and the connection $\tilde{\nabla}$ given by $\tilde{\nabla}:=(1-t)\nabla_0^{TM}+t\nabla_1^{TM}$.
The connection matrix for $\tilde{\nabla}$ is given by:
\[(1-t)\cdot
\begin{bmatrix}
-\bigg[\frac{\partial A}{\partial x}\frac{1}{B}\bigg]_1dy & 0  \\
0 & -\bigg[\frac{dB}{B}\bigg]_1  \\ 
\end{bmatrix}.
\]
The curvature matrix for $\tilde{\nabla}$ is:
\[
\begin{bmatrix}
dt\wedge\bigg[\frac{\partial A}{\partial x}\frac{1}{B}\bigg]_1 dy & 0  \\
0 & dt\wedge\bigg[\frac{dB}{B}\bigg]_1  \\ 
\end{bmatrix}.
\]
The Bott difference form is given by $\pi_* (c_1(\tilde{\nabla}))$ where $\pi_{*}$ is integration
along the fibre of the
projection $\pi:M\times [0,1]\to M$.
The Bott difference form is then:
\[\bigg[\frac{1}{B}\frac{\partial A}{\partial x}\bigg]_1 dy+\bigg[\frac{dB}{B}\bigg]_1.\]
So, the residue for $c_1(\N_{\F,M})$ in $0$ is
\[\frac{1}{2\pi \sqrt{-1}}\int_{\{x=0, |y|=\varepsilon\}} \bigg[\frac{1}{B}\bigg(\frac{\partial
A}{\partial x}+\frac{\partial
B}{\partial y}\bigg)\bigg]_1 dy.\]
\begin{Remark}
Already with slightly harder examples the computations of indices turn out to be really
complicated; please note that while Remark
\ref{Rem:Computation} tells us that in higher dimension the computation follows almost directly
from known results it would be interesting to compute the indices when dealing with singularities
that are not isolated.
\end{Remark}

\section{The residue for the simplest transversal case}
Let $(U_1,x,y)$ be a neighborhood of $0$ in $\mathbb{C}^2$, let $S=\{x=0\}$. Let now $v$ be
a holomorphic
section of
$\T_{M,S(1)}$ with an isolated singularity in $0$. As before, we call $U_0:=U_1\setminus\{0\}$ and
$M=U_1$.
Please remark that we drop the hypothesis about $v$ belonging to $\T_{S(1)}$.
We want to compute the variation index for such a foliation. Since the situation is local we can
assume we have a local
$2$ splitting, first order $\F$-faithful outside $0$ and that we are in a chart adapted to it and
therefore we have a
map $\T_{M,S(1)}$ to $\T_{S(1)}$.
Write $\tilde{v}$ in coordinates as:
\[\tilde{v}=[A(x,y)]_2\frac{\partial}{\partial x}+[B(x,y)]_2\frac{\partial}{\partial y}.\]
Now we can write $[A(x,y)]_2=[\tilde{\rho}([A(x,y)]_2)+R(x,y)]_2$, where $\tilde{\rho}$ is the
$\theta_1$ derivation associated to the $1$-splitting induced by the $2$-splitting; then,
\[\sigma^*(\tilde{v})=(\tilde{\rho}([A(x,y)]_2)\partial/\partial
x+B(x,y)\partial/\partial y.\]
Moreover, we have a splitting $\sigma^*:\T_{M,S}\to\T_S$, givings rise on $U_0\cap S$ to an
isomorphism between $\F_S$, the sheaf of germs of
sections of the
foliation generated by $v_S:=v|_S$ and the sheaf of germs of sections of $\F^{\sigma}.$
Now, the vector field
\[w=[\tilde{\rho}([A(x,y)]_2)]_2\partial/\partial x+[B(x,y)]_2\partial/\partial y\]
is a section of $\T_{S(1)}$, giving rise to a foliation of the first infinitesimal neighborhood.
We can now compute the index as in the former section: the residue for $c_1(N_{\F^{\sigma},M})$ is
therefore:
\begin{align}\frac{1}{2\pi \sqrt{-1}}\int_{\{|y|=\varepsilon\}}&
\bigg[\frac{1}{B}\bigg(\frac{\partial
[\tilde{\rho}([A]_2)]_2}{\partial
x}+\frac{\partial
B}{\partial y}\bigg)\bigg]_1 dy\nonumber\\
&=\frac{1}{2\pi \sqrt{-1}}\int_{\{|y|=\varepsilon\}}\bigg[\frac{1}{B}\bigg(\frac{\partial
}{\partial
x}\bigg(\frac{\partial A}{\partial x}\cdot x\bigg)+\frac{\partial
B}{\partial y}\bigg)\bigg]_1 dy\nonumber\\
&=\frac{1}{2\pi \sqrt{-1}}\int_{\{|y|=\varepsilon\}}
\bigg[\frac{1}{B}\bigg(\frac{\partial A}{\partial x}+\frac{\partial
B}{\partial y}\bigg)\bigg]_1 dy.\nonumber
\end{align}
\begin{Remark}
The term 
\[\frac{\partial^2 A}{\partial x^2}\cdot x\]
in the last computation disappears since it belongs to $\mathcal{I}_S$. 
\end{Remark}

\section{A couple of remarks about extendability of foliations}
In this short section we would like to summarize some of the results of this paper, stressing their
importance towards the understanding of the following problem: when is it possible to extend a holomorphic foliation on
a submanifold $S$ of codimension $m$ in a complex manifold $M$ to a neighborhood of $S$?
Thanks to Theorem \ref{TheoAtiyahSplitting} we know that, if there exists a rank $l$ foliation of the first
infinitesimal neighborhood, if we take any symmetric polynomial $\phi$ of degree larger than $n-m-l$ then
$\phi(\N_{\F,M})$ vanishes.
Therefore, given a foliation $\F$ on $S$, the classes $\phi(\N_{\F,M})$ are obstructions to find an extension to the
first infinitesimal neighborhood, where $\phi$ is a symmetric polynomial of degree larger than $n-m-l$.
In the splitting case we have much more information. As a matter of fact, if the sequence
\[0\to\mathcal{V}\to\tilde{\F}\to\tilde{\F}/\mathcal{V}\to 0\]
splits on the first infinitesimal neighborhood of the zero section of $\N_S$ we know that $\F$ can be extended in a non
involutive way. Therefore, if $S$ splits, the characteristic classes $\phi(\N_{\F,M})$ with $\phi$ is a symmetric
polynomial of degree larger than $n-m-l+\lfloor l/2\rfloor$ are obstructions to find an extension of $\F$ as a non
involutive subbundle of $\T_{S(1)}$.
Known this, if the extension is involutive, also the characteristic classes $\phi(\N_{\F,M})$ with $\phi$ a symmetric
polynomial of degree larger $n-m-l$ and smaller than $n-m-l+\lfloor l/2\rfloor$ vanish.
Therefore, in the splitting case, known that there is a non-involutive extension, the classes $\phi(\N_{\F,M})$ where
$\phi$ is a symmetric polynomial of degree larger $n-m-l$ and smaller than $n-m-l+\lfloor l/2\rfloor$ are
obstructions to find an involutive extension.

Another interesting remark can follow from a simple example; we look at the procedure we built in
section \ref{SecExtension} in the case we have a rank $1$ foliation $\F$ of a codimension $1$ submanifold $S$ in a
complex surface $M$.
We take an atlas adapted to $S$ and $\F$ for the normal bundle $\{V_{\alpha}\}$, supposing \eqref{folest} splits and
also an atlas $\{U_{\alpha}\}$ adapted to $S$ and $\F$ for $M$; now, on the normal bundle in such an atlas we take as
local lifts from $\tilde{\F}/\mathcal{V}$ to $\tilde{\F}$ on each $V_{\alpha}$ the maps 
\[\tau_{\alpha}:\partial_{2,\alpha}\mapsto \frac{\partial}{\partial z^2_{\alpha}},\]
therefore, the obstruction to the splitting of the sequence is represented by
$\{V_{\alpha\beta},\tau_{\beta}-\tau_{\alpha}\}$. Since the sequence splits we know there exists a cochain
$\{V_{\alpha},\sigma_{\alpha}\}$ belonging to
$C^0(\{V_{\alpha}\},\Hom(\tilde{\F}/\mathcal{V},\mathcal{V}))$ such that
$\{V_{\alpha\beta},\sigma_{\beta}-\sigma_{\alpha}\}=\{V_{\alpha\beta},\tau_{\beta}-\tau_{\alpha}\}.$
The isomorphism between $\tilde{F}\simeq \mathcal{V}\oplus \tilde{\F}/\mathcal{V}$ on $S_{N(1)}$ (the first
infinitesimal neighborhood of the embedding of $S$ as the zero section of the normal bundle) is then given on each
$V_{\alpha}$ by
\[(v,w)\mapsto v+\tau_{\alpha} (w)-\sigma_{\alpha}(w).\]
Now, if the map $\sigma_{\alpha}$ is given in coordinates as $c_{\alpha} \cdot\omega^2_{\alpha}\otimes \partial/\partial
v^1_{\alpha}$
the image of $\partial_{2,\alpha}$ is nothing else than
\[\frac{\partial}{\partial z^2_{\alpha}}-c_{\alpha} \frac{\partial}{\partial v^1_{\alpha}}.\]
Therefore, the local generators of the extension to the first infinitesimal neighborhood of the foliation $\F$ on $M$
are given on each $U_{\alpha}$ (modulo rescaling) by:
\[\frac{\partial}{\partial z^2_{\alpha}}-c_{\alpha} \frac{\partial}{\partial z^1_{\alpha}}.\]
As expected, recalling the computation of the residue in section \ref{SecCompResTang} we see that, if $\F$ has an
isolated singular point in $U_{\alpha}$, the computation of the residue depends on the function $c_{\alpha}$. 

All these remarks stress how extending a holomorphic foliation is an important global problem: the results of our paper
show how this problem is strictly connected with the residues and the characteristic classes of $\N_{\F,M}$.

\end{document}